\documentclass[10pt, reqno]{amsart}
\usepackage{lipsum}
\usepackage{a4wide}
\usepackage{amssymb} 
\usepackage{amsmath}
\usepackage{amsthm} 
\usepackage{tikz-cd} 
\usepackage[all]{xy}
\usepackage{amscd}
\usepackage{hyperref}
\usepackage{enumitem}
\hypersetup{colorlinks,linkcolor={blue},citecolor={red},urlcolor={black}}
\theoremstyle{plain}
\usepackage[margin=2.5cm]{geometry}
\numberwithin{equation}{section}

\newcommand{\Ker}{\operatorname{Ker}}

\newtheorem{theorem}{Theorem}[subsection]
\newtheorem{corollary}[theorem]{Corollary}
\newtheorem{lemma}[theorem]{Lemma}
\newtheorem{remark}[theorem]{Remark}
\newtheorem{proposition}[theorem]{Proposition}

\title{Jacquet modules of Tate cohomology and base change lifting}
\author{Sabyasachi Dhar and Santosh Nadimpalli} \date{\today}
\begin{document}
\maketitle
\begin{abstract}
  Let $G$ be a connected reductive group defined over a
  non-Archimedean local field $F$ of residue characteristic $p$. Let
  $\ell$ be a prime number distinct from $p$. Let $E$ be a cyclic
  Galois extension of $F$ with $[E:F]=\ell$.  Let $\Pi$ be a finite
  length $\overline{\mathbb{F}}_\ell$-representation (or an
  $\ell$-modular representation) of $G(E)\rtimes {\rm Gal}(E/F)$.  In
  this context, we prove a conjecture of Treumann and Venkatesh which
  predicts that the Tate cohomology groups
  $\widehat{H}^i({\rm Gal}(E/F), \Pi)$ are finite length
  representations of $G(F)$.  We discuss the explicit computation of
  these Tate cohomology groups when $G$ is ${\rm GL}_n$ and $\Pi$ is
  obtained as a base change lifting of a depth-zero cuspidal
  representation of ${\rm GL}_n(F)$. The primary novelty from our
  previous work is that we treat the case where $\Pi$ is possibly
  non-cuspidal.  We also study the
  ${\rm Gal}(\mathbb{F}_{q^\ell}/\mathbb{F}_q)$-Tate cohomology groups
  of the mod-$\ell$ reduction of the unipotent cuspidal representation
  of ${\rm Sp}_4(\mathbb{F}_{q^\ell})$.
\end{abstract}
\section{Introduction}
Let $G$ be a connected reductive algebraic group defined over a
$p$-adic field $F$. Let $\ell$ be a prime number distinct from $p$.
Let $\sigma$ be an order $\ell$ automorphism of $G$, defined over $F$,
and let $\Gamma$ be the cyclic group generated by $\sigma$.  We denote
by $G^\sigma$ the connected component of the subgroup of
$\sigma$-fixed points of $G$.  Let $\Pi$ be an $\ell$-modular smooth
irreducible representation of $G(F)$ such that $\Pi$ extends to a
representation of $G(F)\rtimes\Gamma$. Treumann and Venkatesh, in
their work on mod-$\ell$ functoriality (\cite[Conjectures
6.3]{Treumann_Venkatesh_linkage}), proposed some remarkable
conjectures relating the $\ell$-modular irreducible smooth
representations of the $p$-adic groups $G(F)$ and $G^\sigma(F)$. Their
conjectures predict that the Tate cohomology groups
$\widehat{H}^i(\Gamma, \Pi)$ are finite length representations of
$G^\sigma(F)$, for $i\in \{0,1\}$, and they also proposed that the
sub-quotients of $\widehat{H}^i(\Gamma, \Pi)$ must be related to $\Pi$
via $\ell$-modular functoriality in the sense of Langlands, and they
called it the {\it linkage principle}. In our previous work
\cite{nadimpalli2024tate}, we proved that $\widehat{H}^i(\Gamma, \Pi)$
is a finite length representation for the case where
$G={\rm Res}_{E/F}{\rm GL}_n$ and $\sigma$ is the automorphism of $G$
induced by a generator of the cyclic group ${\rm Gal}(E/F)$ of order
$\ell$, and we showed that the linkage principle holds when $\ell$
does not divide the pro-order of ${\rm GL}_{n-1}(F)$.

In this article, we first prove the finiteness part of
Treumann--Venkatesh conjecture for the case where $\sigma$ is induced
by a Galois automorphism, i.e., the case where $G$ is
${\rm Res}_{E/F}H$ and $H$ is a connected reductive algebraic group
defined over $F$, the extension $E/F$ is a cyclic Galois extension of
degree $\ell$, and $\sigma$ is induced by a generator of the Galois
group ${\rm Gal}(E/F)$.  The main technique is to control the Jacquet
modules and twisted Jacquet modules of the Tate cohomology groups. We
also use the fact that subrepresentations of finitely generated
$\ell$-modular representations of a $p$-adic group are finitely
generated; this is a deep result proved recently in
\cite{dat_jams}. Let $\Pi$ be the base change lift to ${\rm GL}_n(E)$
of an integral cuspidal irreducible representation $\pi$ of
${\rm GL}_n(F)$ defined over the maximal unramified extension
$\mathcal{K}$ of $\mathbb{Q}_\ell$.  We further show that
$\widehat{H}^0({\rm Gal}(E/F), \mathcal{L})$ is irreducible and a
cuspidal representation, for any
${\rm GL}_n(E)\rtimes {\rm Gal}(E/F)$-stable
$\mathbb{Z}_\ell^{\rm un}$-lattice $\mathcal{L}$ in $\Pi$. In the
depth-zero case, using local methods, we obtain the precise
description of Tate cohomology of a base change lift. These results on
depth-zero representations are obtained by studying the Tate
cohomology of invariant lattices in Shintani lifts for ${\rm GL}_n$
over finite fields.

In two papers \cite{feng2023modular} and \cite{feng2024smith}, Feng
showed that, under mild restrictive hypothesis on the groups involved,
the Genestier--Lafforgue correspondence and the Fargues--Scholze
correspondence are compatible with the linkage principle. However, the
explicit computation of Tate cohomology groups, in particular,
admissibility and the non-vanishing of Tate cohomology
$\widehat{H}^i(\Gamma, \Pi)$, genericity, are not well understood. We
answer some of these questions in this article.  In their paper
\cite[Section 5.0.3]{Henniart_Vigneras_SL2}, Henniart and Vign\'eras
verified Treumann and Venkatesh conjectures for $G(F)={\rm GL}_2(F)$
and $G^\sigma(F) = {\rm SL}_2(F)$ with $\ell =2$.

The following are our principal results.
The first result of this article is to show that the
Tate cohomology groups are finitely generated representations. We
prove the following result.
\begin{theorem}\label{intro_finite_gen}
  Let $G$ be a connected reductive algebraic group over a
  non-Archimedean local field $F$. Let $\Gamma\subset {\rm Aut}(G)(F)$
  be a subgroup of automorphisms of order $\ell$. Let $(\Pi, V)$ be a
  finite length $\ell$-modular representation of
  $G(F)\rtimes\Gamma$. Then the Tate cohomology
  $\widehat{H}^i(\Gamma, \Pi)$ is a finitely generated representation
  of $G^\sigma(F)$.
\end{theorem}
The above theorem is discussed in Section \ref{finte_gen}.
The proof of Theorem \ref{intro_finite_gen} uses the $\ell$-modular version of
Schneider and Stuhler resolutions (\cite[Theorem II.3.1]{Schneider_Stuhler})
of smooth irreducible
representations as in Vign\'eras's work \cite[Proposition 2.6]{vigneras_reso}.
Note that we do not assume anything about $\Gamma$ for the above theorem. 
However, for the next part on the finite length of the Tate cohomology groups, 
we restrict to the Galois automorphism case. 
In what follows, we use the
hypothesis and notations of the above theorem. 

To prove the admissibility of Tate cohomology groups, we 
prove that the Jacquet modules of $\widehat{H}^i(\Gamma, \Pi)$ 
are controlled by the Tate cohomology of Jacquet modules of $\Pi$ 
with respect to $\sigma$-stable parabolic subgroups. This is discussed in 
Section \ref{sec_jac_mod}.
Let $P$ be a $\sigma$-stable parabolic subgroup of $G$. Then the
connected component of the subgroup of $\sigma$-fixed points,
denoted by $P^\sigma$, is a parabolic subgroup of
$G^\sigma$. Let $U$ be the unipotent radical of $P$ and let $N$ be
the unipotent radical of $P^\sigma$. Let $\Psi$ be a character of
$U(F)$. We will show that the natural map of twisted Jacquet modules
\begin{equation}\label{jac_injection_intro}
\widehat{H}^i(\Gamma, \Pi)_{N(F), \Psi}\rightarrow
\widehat{H}^i(\Gamma, \Pi_{U(F), \Psi}),
\end{equation}
is an injection. 
This article, in essence, is about various applications of the above
injective map. First of which is the following theorem. 
\begin{theorem}
Let $G$ be a connected reductive group defined over a non-Archimedean local
field $F$, and let $E/F$ be a cyclic Galois extension of degree
$\ell$. Let $\Pi$ be a finite length $\ell$-modular representation of
$G(E)\rtimes{\rm Gal}(E/F)$. Then the
Tate cohomology group $\widehat{H}^i({\rm Gal}(E/F), \Pi)$ is a finite
length $\ell$-modular representation of $G(F)$.
\end{theorem}
We prove the above theorem using induction on the dimension of $G$.
In the induction argument, we need to control the action of the centre
of $G^\sigma$ along with the exponents in the Jacquet modules with
respect to the parabolic subgroups of $G^\sigma$. For this reason we
restrict to the case where $\sigma$ is induced by a Galois
automorphism.  In the Galois automorphism case, for a Levi subgroup
$M$ defined over $F$, the action of the centre of $M(F)$ on the
${\rm Gal}(E/F)$-Tate cohomology of the corresponding Jacquet modules
of $\Pi$ is via the restriction of the action of the centre of
$M(E)$. In other words, we are using the fact that
$Z(M(E))\cap G(F)=Z(M(F))$. In the general case, i.e., when $\sigma$
is not necessarily induced by a Galois automorphism, we note that
$Z(G^\sigma)$ can be quite large compared to $Z(G)^\sigma$. For
instance, $\sigma$ could be an inner automorphism induced by a regular
element of order $\ell$, and in this case $G^\sigma$ is a torus.

For the next result, proved in Section \ref{sec_gen_irr}, we assume
that $G$ is ${\rm Res}_{E/F} {\rm GL}_n$, and $\sigma$ is an
automorphism of $G$ induced by a generator of the Galois group
${\rm Gal}(E/F)$. Let $\mathcal{K}$ be the maximal unramified
extension of $\mathbb{Q}_\ell$ and $\Lambda$ be the ring of integers
of $\mathcal{K}$.  Let $\Pi$ be an absolutely irreducible, integral
representation of ${\rm GL}_n(E)$, defined over $\mathcal{K}$ such
that $\Pi^\sigma\simeq \Pi$.  Let $\pi$ be an irreducible $\ell$-adic
integral cuspidal representation of ${\rm GL}_n(F)$ such that
$\Pi\otimes_\mathcal{K}\overline{\mathbb{Q}}_\ell$ is the base change
lifting of $\pi$. Note that $\Pi$ is not necessarily a cuspidal
representation. We show the following result on any
$\Lambda[{\rm GL}_n(E)\rtimes {\rm Gal}(E/F)]$ stable lattice in
$\Pi$.
\begin{theorem}\label{intro_gln}
With the Hypothesis on $\Pi$ and $\pi$ as above, 
  let $\mathcal{L}$ be a $\Lambda[{\rm GL}_n(E)\rtimes {\rm Gal}(E/F)]$ stable
  lattice in $\Pi$. Then, the Tate cohomology
  $\widehat{H}^0({\rm Gal}(E/F), \mathcal{L})$ is an $\ell$-modular irreducible
  cuspidal representation of ${\rm GL}_n(F)$ and 
  $\widehat{H}^1({\rm Gal}(E/F), \mathcal{L})$ is trivial. 
\end{theorem}
In the proof of the above result, we first prove that the
${\rm Gal}(E/F)$-Tate cohomology of a
$\Lambda[{\rm GL}_n(E)\rtimes {\rm Gal}(E/F)]$ stable lattice in $\Pi$
has a unique generic sub-quotient. In fact, this genericity result
holds in greater generality. Let $G$ be a split connected reductive
group defined over $F$. Let $\Pi$ be a smooth irreducible
$\ell$-modular representation of $G(E)$ such that $\Pi$ is stable
under the action of ${\rm Gal}(E/F)$. Let $B$ be a Borel subgroup of
$G$ defined over $F$, and let $U$ be the unipotent radical of $B$. Let
$\psi:U(E)\rightarrow \overline{\mathbb{F}}_\ell^\times$ be a
non-degenerate, ${\rm Gal}(E/F)$ invariant character such that $\psi$
restricts to a non-degenerate character on $U(F)$. If $\Pi$ is
$(U(E), \psi)$-generic, then we show that
$\widehat{H}^0({\rm Gal}(E/F), \Pi)$ has a unique
$(U(F), \psi)$-generic sub-quotient (Theorem \ref{Tate_generic}). Then
to prove Theorem \ref{intro_gln}, we use the injection
\eqref{jac_injection_intro} and prove the right-hand side of the
equation \eqref{jac_injection_intro} is trivial by using the fact that
the non-trivial Jacquet modules of a non-cuspidal base change lifting
$\Pi$ are multiplicity free and their Jordan Holder factors are not
stable under ${\rm Gal}(E/F)$.

In Section \ref{sec_shintani}, we then study the Tate cohomology of
invariant lattices in representations obtained as Shintani liftings
(to ${\rm GL}_n(\mathbb{F}_{q^\ell})$) of representations of finite
group ${\rm GL}_n(\mathbb{F}_q)$.  The analogue of Theorem
\ref{intro_gln} is valid for $\Pi$ which is obtained as a Shintani
lift to ${\rm GL}_n(\mathbb{F}_{q^\ell})$ of a cuspidal representation
of ${\rm GL}_n(\mathbb{F}_q)$; we use the Deligne--Lusztig character
formula to show that, for any
$\Lambda[{\rm GL}_n(\mathbb{F}_{q^\ell})\rtimes{\rm
  Gal}(\mathbb{F}_{q^\ell}/\mathbb{F}_q)]$ stable lattice
$\mathcal{L}$ in $\Pi$, the space
$\widehat{H}^0({\rm Gal}(\mathbb{F}_{q^{\ell}}/\mathbb{F}_q),
\mathcal{L})$ is isomorphic to the Frobenius twist of the mod-$\ell$
reduction of $\pi$ (Theorem \ref{TV_finite_field}).

We end this article by considering the example of $\theta_{10}$ in
Section \ref{sec_theta10}, the unipotent cuspidal irreducible
representation of the symplectic group ${\rm
  Sp}_4(\mathbb{F}_{q})$. Let $\pi_{n}$ be the unique cuspidal
unipotent representation of ${\rm Sp}_4(\mathbb{F}_{q^n})$. Note that
the mod-$\ell$ reduction of $\pi_n$ is irreducible and it is denoted
by $r_\ell(\pi_n)$.  For a vector space $W$ over the field
$\overline{\mathbb{F}}_\ell$, we denote by $W^{(\ell)}$ the
$\overline{\mathbb{F}}_\ell$-vector space where the scalar action is
twisted by the Frobenius automorphism of the field
$\overline{\mathbb{F}}_\ell$. We have
\begin{theorem}
Let $\ell$ be an odd prime number. 
Let $\pi_n$ be the unipotent cuspidal representation of ${\rm
Sp}_4(\mathbb{F}_{q^n})$. Then, we have the following isomorphism of
${\rm Sp}_4(\mathbb{F}_q)$ representations:
$$\widehat{H}^0({\rm
Gal}(\mathbb{F}_{q^\ell}/\mathbb{F}_q),
r_\ell(\pi_\ell))\simeq r_\ell(\pi_1)^{(\ell)}.$$
Moreover, the Tate cohomology $\widehat{H}^1({\rm
Gal}(\mathbb{F}_{q^\ell}/\mathbb{F}_q), 
r_\ell(\pi_\ell))$ is
trivial. 
\end{theorem}
The above theorem seems to suggest that 
for all classical groups $G$ admitting a cuspidal
unipotent representation, the zeroth Tate cohomology group for the
action of ${\rm Gal}(\mathbb{F}_{q^\ell}/\mathbb{F}_q)$ on the
mod-$\ell$ reduction of a cuspidal unipotent representation of
$G(\mathbb{F}_{q^\ell})$ is isomorphic to the Frobenius twist of the
mod-$\ell$ reduction of the cuspidal unipotent representation of
$G(\mathbb{F}_q)$. 
\subsection*{{\bf Notations}}
\begin{enumerate}
    \item We use the standard notations attached to
any non-Archimedean local field $F$: the ring of integers of  $F$ is
$\mathfrak{o}_F$, the maximal ideal of $\mathfrak{o}_F$ is
$\mathfrak{p}_F$, and the residue field
$\mathfrak{o}_F/\mathfrak{p}_F$ is $k_F$. The field $k_F$ has $q_F$
elements and $q_F$ is a power of a prime $p$. 
\item Let $G$ be a connected
reductive algebraic group defined over $F$.
Let
$\sigma\in {\rm Aut}(G)$, defined over $F$, be an element of prime order $\ell$ and
$\ell\neq p$.  Let $\Gamma$ be the group generated by $\sigma$. 
Let $G^\sigma$ be the connected component of the fixed points of $\sigma$,
and it is assumed to be reductive. 
\item A
representation $(\pi, V)$ of $G(F)$ is said to be smooth if every
vector $v\in V$ is fixed by a compact open subgroup of $G(F)$. Let $P$
be an $F$-rational parabolic subgroup of $G$ and the unipotent 
radical of $P$ is denoted by $U$. Let $M$ be a Levi
subgroup of $P$ defined over $F$. For a smooth representation $(\tau, W)$ of
$M(F)$, we denote by $i_{P(F)}^{G(F)}(\tau)$ the normalised parabolic
induction of $\tau$ to $G(F)$. We use the notation
$r_{U(F)}(\pi)$ for the normalised Jacquet module of a smooth
representation $(\pi, V)$ of $G(F)$.
\item  In this article, a smooth
representation $(\pi, V)$ of $G(F)$ is called cuspidal if
$r_{U(F)}(\pi)=0$, where $U$ is the unipotent radical of any proper
$F$-rational parabolic
subgroup $P$ of $G$ . For a closed subgroup $H_1$ of a reductive
$p$-adic group $H$, we will denote by ${\rm Ind}_{H_1}^H$ and
${\rm ind}_{H_1}^H$ the smooth induction functor and the compact
induction functor, respectively.
\end{enumerate}   

\section{Tate cohomology of smooth representations}\label{finte_gen}
Let $\Pi$ be a smooth finite length representation of $\overline{\mathbb{F}}
_\ell[G(F)\rtimes \Gamma]$. We prove that the Tate
cohomology groups $\widehat{H}^i(\Gamma, \Pi)$ are finitely generated
representations of $G^\sigma(F)$. We use Schneider--Stuhler
resolutions and their mod-$\ell$ version by Vign\'eras for proving this
finiteness statement. 
\subsection{}
Fix an algebraic closure $\overline{\mathbb{Q}}_\ell$ of
$\mathbb{Q}_\ell$ and let $\mathcal{K}$ be the maximal unramified
extension of $\mathbb{Q}_\ell$ in $\overline{\mathbb{Q}}_\ell$. Let
$\Lambda$ be the ring of integers of $\mathcal{K}$. All
$\Lambda[G(F)]$ modules are assumed to be smooth, i.e., every element
is fixed by an open subgroup of $G(F)$. 
Let $M$ be a
$\Lambda[G(F)\rtimes \Gamma]$ module. Let ${\rm Nr}$ be the norm
element $\sum_{g\in \Gamma}g\in \Lambda[\Gamma]$.  Consider the 
 chain complex of $\Lambda[\Gamma]$ modules
$$ \cdots\longrightarrow M\xrightarrow{{\rm id}-\sigma}
M\xrightarrow{{\rm Nr}} M\xrightarrow{{\rm id}-\sigma}
M\xrightarrow{{\rm Nr}} M\longrightarrow\cdots $$ The Tate cohomology
groups of $M$, with respect to the action of $\Gamma$, are
$2$-periodic and we have
$$ \widehat{H}^0(\Gamma, M) := \dfrac{{\rm Ker}({\rm id}-\sigma)}
{{\rm Img}({\rm Nr})},\,\,\, \widehat{H}^1(\Gamma, M) := \dfrac{{\rm
    Ker}({\rm Nr})} {{\rm Img}({\rm id}-\sigma)}. $$ The action of
$G(F)$ on $M$ induces an action of $G^\sigma(F)$ on the Tate
cohomology groups $\widehat{H}^i(\Gamma, M)$, and we get two
$\ell$-modular smooth representations $\widehat{H}^i(\Gamma, M)$, for
$i\in \{0,1\}$, of the fixed point subgroup $G^\sigma(F)$. Similar
cohomology groups can be defined when $F$ is a finite field, and we
will assume the easy modifications in the definitions.
\subsection{}
\label{section_kirrilov_rep}
Let $X$ be a Hausdorff, locally compact, and totally disconnected
topological space (an $\ell$-space in the sense of \cite[Section
1.1]{BZ_0}). Let $G$ be a locally profinite group with an order $\ell$
automorphism $\sigma$, generating a group $\Gamma$ of automorphisms of
$G$. Assume that $G\rtimes\Gamma$ acts continuously on $X$. Let
$\mathcal{F}$ be a $G\rtimes\Gamma$ equivariant sheaf on $X$.  The
natural restriction map
$$\widehat{H}^i(\Gamma, H^0(X, \mathcal{F}))\rightarrow H^0(X^\sigma,
\widehat{H}^i(\Gamma, \mathcal{F}_{|X^\sigma}))$$ is an isomorphism,
for $i\in \{0,1\}$ (see \cite[Proposition
3.3]{Treumann_Venkatesh_linkage}).  Hear, $H^0(X, \mathcal{F})$ is the
space of global sections of a sheaf $\mathcal{F}$ on $X$.  Let $H$ be
a closed subgroup of $G$ such that $\sigma(H) = H$. Let $V$ be a
$\Lambda[H\rtimes \Gamma]$ module. This data defines a
$G\rtimes\Gamma$ equivariant sheaf $\mathcal{F}_V$ on $G/H$ whose
sections on an open set $U$ of $G/H$ are the sections of the vector
bundle $G\times_HV\rightarrow G/H$ restricted to $U$. The global
sections of the sheaf $\mathcal{F}_V$ is the induced representation
${\rm ind}_H^GV$, and the action of $\Gamma$ is given by
$$ (\sigma.f)(g)= \sigma\big(f(\sigma^{-1}(g))\big), $$
for all $g\in G$ and $f\in {\rm ind}_H^GV$. 
Assume that $(G/H)^\sigma=G^\sigma/H^\sigma$. 
Then, the restriction to $G^\sigma$ map induces a $G^\sigma$
equivariant isomorphism
\begin{equation}\label{Tate_restriction_map}
\widehat{H}^i(\Gamma, {\rm ind}_H^GV) \rightarrow
{\rm ind}_{H^\sigma}^{G^\sigma}
(\widehat{H}^i(\Gamma, V)),\,\,\, f\mapsto {\rm res}_{G^\sigma}(f)
\end{equation} 
\subsection{}
Let $(\Pi,V)$ be a smooth finite length representation of
$\overline{\mathbb{F}}_\ell[G(F)\rtimes\Gamma]$.  We show that $\widehat{H}^i(\Gamma, \Pi)$ is a finitely generated
$G^\sigma(F)$ representation. We use the mod-$\ell$ version of
Schneider and Stuhler resolutions (\cite[Theorem
II.3.1]{Schneider_Stuhler}) of smooth irreducible representations
as discussed in Vign\'eras (\cite[Proposition 2.6]{vigneras_reso}). Let
$\mathfrak{X}$ be the extended Bruhat--Tits building of a connected
reductive algebraic group $G$ defined over a non-Archimedean local field $F$.
Let $d$ be the $F$-rank of $G$. 
For $0\leq k\leq d$, let $\mathfrak{X}_{k}$ be the set of $k$-dimensional 
facets of $\mathfrak{X}$, and
let $\mathfrak{X}_{(k)}$ be the set of oriented $k$-facets which are
tuples of the form $(\mathcal{F}, c_\mathcal{F})$, where
$\mathcal{F}\in \mathfrak{X}_k$ and $c_\mathcal{F}$ is an orientation
of $\mathcal{F}$. Let $e\geq 0$. For a facet $\mathcal{F}$ of
$\mathfrak{X}$, Schneider and Stuhler associate a compact open
subgroup $U_\mathcal{F}^{(e)}$, and let $\underline{V}$ be the
coefficient system defined by associating every facet $\mathcal{F}$ to
$V^{U^e_\mathcal{F}}$. Let $C(\mathfrak{X}_{(k)}, \underline{V})$ be
the set of oriented cellular $k$-chains on $\mathfrak{X}_{(k)}$. Then 
we have a $G$ equivariant resolution of the representation
\begin{equation}\label{ss_resolu}
0\rightarrow C(\mathfrak{X}_{(d)}, \underline{V})
\xrightarrow{\partial_{d-1}} C(\mathfrak{X}_{(d-1)},
\underline{V})\xrightarrow{\partial_{d-2}} \cdots \rightarrow
C(\mathfrak{X}_{(1)}, \underline{V})\xrightarrow{\partial_0}
C(\mathfrak{X}_{(0)}, \underline{V})\xrightarrow{\epsilon}
V\rightarrow 0
\end{equation}
for some $e\geq 0$. We refer to \cite[Section II.2]{Schneider_Stuhler}
for details. 
Note the automorphism $\sigma$ acts on $\mathfrak{X}$ simplicially,
i.e., it preserves the sets $\mathfrak{X}_{d}$ and $\mathfrak{X}_{(d)}$.
\begin{lemma}\label{finite_gen_1}
For $0\leq k\leq d$,
the Tate cohomology groups $\widehat{H}^i(\Gamma, C(\mathfrak{X}_{(k)},
\underline{V}))$ are finitely generated
$\overline{\mathbb{F}}_\ell[G^\sigma(F)]$ modules. 
\end{lemma}
\begin{proof}
  Consider the set $\mathfrak{X}_{(k)}$ with discrete topology, and
  the action of the group $G(F)$ on $\mathfrak{X}_{(k)}$ is
  continuous. Also note that the sheaf $\underline{V}$ is $G$ and
  $\Gamma$ equivariant.  Since $V$ is an admissible
  representation of $\overline {\mathbb{F}}_\ell[G(F)]$, the stalks of
  the sheaf $\underline{V}$ are all finite dimensional vector spaces
  over $\overline{\mathbb{F}}_\ell$. Using \cite[Proposition
  3.3]{Treumann_Venkatesh_linkage}, we get that the restriction map
$$ \widehat{H}^i(\Gamma, C(\mathfrak{X}_{(k)}, 
\underline{V}))\rightarrow C(\mathfrak{X}_{(k)}^\sigma,
\widehat{H}^i(\Gamma, \underline{V})) $$ is an isomorphism. Let
$\mathcal{Y}$ be the Bruhat--Tits building of $G^\sigma$, then
$\mathfrak{X}^\sigma_{k}$ is a subset of facets of
$\mathcal{Y}$. Thus, there are only finitely many orbits for the
action of $G^\sigma$ on $\mathfrak{X}^\sigma_{(k)}$. Since the module
$\widehat{H}^i(\Gamma, V^{U_F^e})$ is finite, we get that the space
$C(\mathfrak{X}_{(k)}^\sigma, \widehat{H}^i(\Gamma, \underline{V})) $ is a
finitely generated representation of
$\overline{\mathbb{F}}_\ell[G^\sigma(F)]$. This shows the lemma.
\end{proof}
In what follows, we need the fact that the submodules of finitely generated
smooth $\overline{\mathbb{F}}_\ell[H(F)]$  representations are again finitely
generated.  This is a fundamental result in the theory of smooth
complex representations due to Bernstein and in the mod-$\ell$ case it
follows from the deep results of the paper \cite[Corollary
1.3]{dat_jams} seen with the work of Dat (\cite[Corollary
4.5]{Dat_finiteness}).
\begin{lemma}\label{finite_gen_2]}
Let $(\Pi,V)$ be an irreducible smooth representation of
$\overline{\mathbb{F}}_\ell[G(F)]$. 
Assume that $(\Pi,V)$ extends as a
representation of $\overline
{\mathbb{F}}_\ell[G(F)\rtimes \Gamma]$. Then the Tate
cohomology groups $\widehat{H}^i(\Gamma, V)$ 
are finitely generated
representations of $\overline
{\mathbb{F}}_\ell[G^\sigma(F)]$.
\end{lemma}
\begin{proof}
Note that the resolution
\eqref{ss_resolu} is $G$ and $\Gamma$
equivariant, and the Tate cohomology groups
$$\widehat{H}^i(\Gamma, C(\mathfrak{X}_{(i)}, \underline{V})),\ \,\, 
0\leq i\leq d,$$ are finitely generated representations of
$\overline{\mathbb{F}}_\ell[G^\sigma(F)]$ by Lemma
\ref{finite_gen_1}. For each $2\leq k\leq d$, consider the short exact sequence of
$\overline{\mathbb{F}}_\ell[G(F)\rtimes \Gamma]$ modules
$$ 0\rightarrow {\rm Im}(\partial_{k-1}) \rightarrow C(\mathfrak{X}_{(k-1)},
\underline{V}) \rightarrow {\rm Ker}(\partial_{k-2}) \rightarrow 0. $$
Then the corresponding long exact sequence of Tate cohomology groups gives
the following hexagonal exact sequence
$$
\xymatrixrowsep{0.15in} \xymatrixcolsep{0.15in} \xymatrix{
  \widehat{H}^0(\Gamma, {\rm Im}(\partial_{k-1})) \ar[rr]^{} &&
  \widehat{H}^0(\Gamma, C(\mathfrak{X}_{(k-1)},
  \underline{V})) \ar[rr] && \widehat{H}^0(\Gamma, {\rm Ker}(\partial_{k-2})) \ar[dd] \\\\
  \widehat{H}^1(\Gamma, {\rm Ker}(\partial_{k-2})) \ar[uu] &&
  \widehat{H}^1(\Gamma, C(\mathfrak{X}_{(k-1)}, \underline{V}))
  \ar[ll] && \widehat{H}^1(\Gamma, {\rm Im} (\partial_{k-1})) \ar[ll] }
$$
which we denote by $S(k)$. Note that
${\rm Im}(\partial_{d-1})$ is
identified with $C(\mathfrak{X}_{(d)},
\underline{V})$. Since both the Tate cohomology
$\widehat{H}^i(\Gamma, C(\mathfrak{X}_{(d)},
\underline{V}))$ and
$\widehat{H}^i(\Gamma, C(\mathfrak{X}_{(d-1)},
\underline{V}))$ are finitely
generated, it follows from the diagram 
$S(d)$ that $\widehat{H}^i(\Gamma, {\rm Ker}
(\partial_{d-2}))$ is also finitely
generated. A similar hexagonal 
diagram of Tate cohomology
corresponding to the following 
short exact sequence of $G(F)\rtimes\Gamma$ modules
$$ 0\rightarrow {\rm Ker}(\partial_{d-2})
\rightarrow C(\mathfrak{X}_{(d-1)}, \underline{V}) \rightarrow {\rm
  Im}(\partial_{d-2}) \rightarrow 0 $$ shows that
$\widehat{H}^i(\Gamma, {\rm Im} (\partial_{d-2}))$ is finitely
generated. Since
$\widehat{H}^i(\Gamma, C(\mathfrak{X}_{(d-2)}, \underline{V}))$ is
finitely generated, it follows from the hexagonal diagram $S(d-1)$
that $\widehat{H}^i(\Gamma, {\rm Ker}(\partial_{d-3}))$ is finitely
generated. Thus, following a finite number of inductive steps, the
hexagonal diagram $S(2)$ shows that the Tate cohomology
$\widehat{H}^i(\Gamma, {\rm Ker}(\partial_0))$ and
$\widehat{H}^i(\Gamma, {\rm Im}(\partial_0))$ are finitely
generated. Finally, it follows from the hexagonal diagram of Tate
cohomology groups corresponding to the following short exact sequence
of $\overline{\mathbb{F}}_\ell[G(F)\rtimes\Gamma]$ modules
$$ 0\rightarrow {\rm Im}(\partial_{0})
\rightarrow
C(\mathfrak{X}_{(0)},\underline{V})
\xrightarrow{\epsilon} V\rightarrow 0, $$
that the Tate cohomology
$\widehat{H}^i(\Gamma, V)$ is finitely
generated as a representation of
$\overline{\mathbb{F}}_\ell[G^\sigma(F)]$.
\end{proof}
\begin{theorem}\label{Tate_finite_gen}
Let $(\Pi, V)$ be a finite length
representation of
$\overline{\mathbb{F}}_\ell[G(F)\rtimes \Gamma]$. Then
each Tate cohomology group $\widehat{H}^i(\Gamma, V)$
is a finitely generated
representation of $\overline
{\mathbb{F}}_\ell[G^\sigma(F)]$.
\end{theorem}
\begin{proof}
We use induction on the length
of $(\Pi,V)$ as an $\overline{\mathbb{F}}_\ell[G(F)]$
representation. If
$(\Pi,V)$ is an irreducible
$G(F)$ representation, then
the theorem follows from Lemma
\ref{finite_gen_2]}. Let $W$
be an irreducible $G(F)$
subrepresentation of $V$. 
First consider the case when $W$ is $\Gamma$
stable. The long exact sequence of
Tate cohomology associated with the
following short exact sequence
$$ 0\longrightarrow W\longrightarrow V
\longrightarrow V/W\longrightarrow 0  $$
gives two exact sequences
$$ \widehat{H}^0(\Gamma, W)\longrightarrow
\widehat{H}^0(\Gamma, V)\longrightarrow
\widehat{H}^0(\Gamma, V/W) $$
and
$$ \widehat{H}^1(\Gamma, W)\longrightarrow
\widehat{H}^1(\Gamma, V)
\longrightarrow \widehat{H}^1(\Gamma, V/W). $$ 
Since $W$ is irreducible, each Tate cohomology
$\widehat{H}^i(\Gamma, W)$ is a finitely generated
representation of
$\overline{\mathbb{F}}_\ell[G^\sigma(F)]$ by Lemma
\ref{finite_gen_2]}. Also note that
the length of the $G(F)$ representation $V/W$ is strictly
less than the length of $V$, inductively we may assume that the
Tate cohomology groups $\widehat{H}^i(\Gamma, V/W)$
are finitely generated representations of $G^\sigma(F)$. Now,
it follows from the above exact
sequences that each $\widehat{H}^i(\Gamma, V)$
is a finitely generated representation of
$\overline{\mathbb{F}}_\ell[G^\sigma(F)]$.
Let us now consider the case where $W$ is not $\Gamma$ stable. Then
$$ W' = \bigoplus_{i=0}^{\ell-1} W^{\sigma^i} $$
is a $G(F)\rtimes\Gamma$
stable subspace of $V$, where
the action of $G(F)$ on the space
$W^{\sigma^i}$ is twisted by
$\sigma^i$ for  $0\leq i<\ell$. Note that 
$\widehat{H}^i(\Gamma, W')$ is
trivial, for $i\in \{0,1\}$. Now, the long exact
sequence of Tate cohomology groups applied 
to the following short exact sequence of
$\overline{\mathbb{F}}_\ell[G(F)\rtimes\Gamma]$ modules
$$ 0\longrightarrow W'\longrightarrow
V \longrightarrow V/W'\longrightarrow 0 $$ 
gives the following isomorphism of $G^\sigma(F)$ representations
$$ \widehat{H}^i(\Gamma, V)\rightarrow
\widehat{H}^i(\Gamma, V/W'),\ i\in \{0,1\}. $$
Since the length of $V/W'$ is
strictly less than that of $V$, the $G^\sigma(F)$
representation $\widehat{H}^i(\Gamma, V/W')$
is finitely generated by induction. Thus, we get that the Tate cohomology group
$\widehat{H}^i(\Gamma, V)$ is a finitely generated representation of 
$G^\sigma(F)$.
\end{proof}

\section{Finiteness of Tate cohomology, the Galois automorphisms
  case}\label{sec_jac_mod}
In this section, we prove some results on the compatibility of Jacquet
modules and twisted Jacquet modules with Tate cohomology groups. Some of 
these results were proved in \cite[Section 6]
{nadimpalli2024tate} for the group $G={\rm Res}_{E/F}{\rm GL}_n$,
where $E/F$ is a Galois extension of prime degree $\ell$, and $\sigma$
is an automorphism induced by a generator of the Galois group
${\rm Gal}(E/F)$. The results in this section, although
proved for groups defined over non-Archimedean local fields,
are valid for finite
groups of Lie type. For convenience, we will drop $\Gamma$ in the Tate cohomology
groups $\widehat{H}^i(\Gamma, M)$ and simply denote these cohomology spaces by
$\widehat{H}^i(M)$. 
\subsection{Twisted Jacquet modules}
Let $(\pi,V)$ be a smooth $\Lambda[G(F)]$ module, where $\Lambda$ is
the ring of integers of the maximal unramified extension of
$\mathbb{Q}_\ell$ in $\overline{\mathbb{Q}}_\ell$.  Let $P=MN$ be an
$F$-rational parabolic subgroup of $G$, where $M$ is an $F$-rational
Levi subgroup of $P$ and 
$N$ is the unipotent radical of $P$. Let
$\Theta:N(F)\rightarrow \Lambda^\ast$ be a character. Let
$V(N(F),\Theta)$ be the $\Lambda$-module generated by the following
set of vectors
$$ \big\{\pi(n)v-\Theta(n)v:v\in V,\,
n\in N(F)\big\}. $$ The quotient space $V/V(N(F),\Theta)$ is called a
twisted Jacquet module of $V$, which is denoted by $V_{N(F),\Theta}$.
This space is characterised by the set of elements $v\in V$ for which
there exists some compact open subgroup $\mathcal{N}_v$ of  $N(F)$ such
that
$$ \int_{\mathcal{N}_v} \Theta^{-1}(n)\,
\pi(n)\,dn = 0, $$ where $dn$ is a Haar measure on $N(F)$.  
The space $V_{N(F),\Theta}$ is a smooth
representation of the $M(F)$ stabiliser of 
the character $\Theta$, to be denoted by $M(F)_\Theta$.  In particular,
when $\Theta=1$, the trivial character of $N(F)$, then the smooth
$\Lambda[M(F)]$ module $V_{N(F),1}$ is called the Jacquet module of
$V$. We also use the notation $r_{N(F)}(V)$ for $V_{N(F),1}$.
\subsection{Integral representation and mod-\texorpdfstring{$\ell$}{}
reduction}
A smooth $\ell$-adic representation $(\pi,V)$ of $G(F)$, i.e., $V$ is
a vector space over $\overline{\mathbb{Q}}_\ell$, is called integral
if $V$ is a finite length $G(F)$ representation and if there exists a
$G(F)$ stable $\overline{\mathbb{Z}}_\ell$-lattice $\mathcal{L}$ in
$V$. Suppose $(\pi,V)$ is defined over $\mathcal{K}$, i.e., there
exists a $G(F)$ stable $\mathcal{K}$-vector subspace $V_0$ in $V$ such
that $V_0\otimes_{\mathcal{K}} \overline{\mathbb{Q}}_\ell = V$, then
there exists a $\Lambda[G(F)]$ stable lattice $\mathcal{L}_0$ in $V_0$
and $\mathcal{L}_0 \otimes_{\Lambda}\overline{\mathbb{Z}}_\ell$ is a
$G(F)$ stable $\overline {\mathbb{Z}}_\ell$-lattice in $V$.  If $P$ is
an $F$-rational parabolic subgroup of $G$ with an $F$-rational Levi
subgroup $M$ and unipotent radical $N$, then it follows from
\cite[Proposition 1.4]{Dat_tempered} that $\mathcal{L}_{N(F)}$ is a
$M(F)$ stable $\overline {\mathbb{Z}}_\ell$-lattice in $V_{N(F)}$.
The natural action of $G(F)$ on the
$\overline{\mathbb{F}}_\ell$-vector space
$\mathcal{L}\otimes_{\overline{\mathbb{Z}}_\ell}
\overline{\mathbb{F}}_\ell$ gives an $\ell$-modular representation of
$G(F)$, which depends on the choice of the lattice
$\mathcal{L}$. Using \cite[Chapter II, 5.11.a and
5.11.b]{Vigneras_modl_book}, the representation
$(\pi,\mathcal{L}\otimes_ {\overline{\mathbb{Z}}_\ell}
\overline{\mathbb{F}}_\ell)$ has finite length and its
semi-simplification is independent of the choice of the $G(F)$ stable
$\overline{\mathbb{Z}}_\ell$-lattice $\mathcal{L}$ in $V$. The
semi-simplification of
$(\pi, \mathcal{L}\otimes_{\overline{\mathbb{Z}}_\ell}
\overline{\mathbb{F}}_\ell)$, denoted by $r_\ell(\pi)$, is called the
mod-$\ell$ reduction of $\pi$.
\subsection{}
Let $\sigma\in {\rm Aut}(G)$, defined over $F$, be an element of order
$\ell$, where $\ell$ is a prime different from $p$. Let $\Gamma$ be
the group generated by $\sigma$. Let $P$ be a $\sigma$-stable
$F$-parabolic subgroup of $G$, and let $N$ be the unipotent radical of
$P$. Note that $N$ is $\sigma$-stable and let $M$ be a $\sigma$-stable
Levi subgroup $P$ (for the existence of $M$, see \cite[Lemma
12.5.4]{Prasad_Kaletha_Bruhat_Tits} for a reference).  Given any
$F$-parabolic subgroup $Q$ of $G^\sigma$, there exists a
$\sigma$-stable $F$-parabolic subgroup $P$ of $G$ such that
$Q=P^\sigma$ (see \cite[Lemma 12.5.4]{Prasad_Kaletha_Bruhat_Tits} for
a reference). Note that $U=N^\sigma$ is the unipotent radical of $Q$
and $M^\sigma$ is a Levi subgroup of $Q$.  Let
$\Theta:N(F)\rightarrow \Lambda^\ast$ be a character such that
$\Theta^\sigma = \Theta$.
\begin{lemma}\label{Tate_image}
Let $(\pi,\mathcal{L})$ be a smooth
$\Lambda[G(F)\rtimes \Gamma]$ module.
For each $i\in\{0,1\}$, the image of the natural map
$\widehat{H}^i(\mathcal{L}(N(F), \Theta))\rightarrow
\widehat{H}^i(\mathcal{L})$ is equal to
 $\widehat{H}^i(\mathcal{L}) (N^\sigma(F),\overline{\Theta})$, where
 $\overline{\Theta}$ is the the mod-$\ell$ reduction of $\Theta$.
\end{lemma}
\begin{proof}
For $i\in\{0,1\}$, let $\varphi_i$ be the natural map
$$ \widehat{H}^i(\mathcal{L}(N(F),
\Theta))\longrightarrow \widehat{H}^i(\mathcal{L}). $$ We first
consider the case $i=0$. Let $v$ be in the image of $\varphi_0$, and
let $\tilde{v}$ be a lift of $v$ in $\mathcal{L}(N(F),\Theta)^\sigma$.
There exists a compact open subgroup $\mathcal{N}\subseteq N(F)$ such
that
\begin{equation}\label{lift_sum_1}
\int_{\mathcal{N}}\Theta^{-1}(n)\,\pi(n)\tilde{v}\,dn = 0.
\end{equation}
Since $(\pi,\mathcal{L})$ is smooth, there exists a compact open
subgroup $\mathcal{N}'\subseteq \mathcal{N}$ of finite index such that
$\mathcal{N'}$ fixes the vector $\tilde{v}$ and $\Theta$ is trivial on
$\mathcal{N}'$. Then
$$ \int_{\mathcal{N}}\Theta^{-1}(n)\,
\pi(n)\tilde{v}\,dn = \sum_{n\in \mathcal{N}/\mathcal{N}'}
\Theta^{-1}(n)\,\pi(n)\tilde{v}\,dn. $$ Since $N(F)$ admits a
filtration of $\Gamma$ stable compact open subgroups, we
assume that both $\mathcal{N}$ and $\mathcal{N}'$ are $\Gamma$-stable. If $X$
denotes the coset space $\mathcal{N} /\mathcal{N}'$, then we have
\begin{equation}\label{sum_lift_2}
\sum_{x\in X} \Theta^{-1}(x)\,\pi(x)\tilde{v} =
\sum_{y\in X^\sigma} \Theta^{-1}(y)\,\pi(y)\tilde{v}\ +
\sum_{z\in {X\setminus X^\sigma}} \Theta^{-1}(z)\,\pi(z)\tilde{v}.
\end{equation}
Since the action of $\Gamma$ on $X\setminus X^\sigma$ is
free, there exists a subset $\mathcal{U}\subseteq X\setminus X^\sigma$
such that $X\setminus X^\sigma$ is the disjoint union of
$\sigma^i\mathcal{U}$ for $1\leq i\leq\ell$. As $\tilde{v}$ is a
$\sigma$-fixed vector, we get
$$ \sum_{z\in {X\setminus X^\sigma}} \Theta^{-1}(z)\,\pi(z)\tilde{v} =
\sum_{i=1}^\ell\sum_{u\in \mathcal{U}}\Theta^{-1}(\sigma^iu)\,
\pi(\sigma^iu)\tilde{v} = {\rm Nr}\big(\sum_{u\in \mathcal{U}}
\Theta^{-1}(u)\,\pi(u)\tilde{v}\big), $$ where
${\rm Nr} = 1+\sigma+\cdots +\sigma^{\ell-1}\in \Lambda[\Gamma]$. This
shows that
$$ \sum_{z\in {X\setminus X^\sigma}} 
\overline{\Theta}^{-1}(z)\,\pi(z)v = 0 $$ 
in $\widehat{H}^0(\mathcal{L})$. Then it follows from
(\ref{lift_sum_1}) and (\ref{sum_lift_2}) that
$$ \sum_{y\in X^\sigma} 
\overline{\Theta}^{-1}(y)\,\pi(y)v = 0. $$ This implies that the vector $v$
belongs to the Tate cohomology space
$\widehat{H}^0(\mathcal{L})(N^\sigma(F),\overline{\Theta})$.
Conversely, let $w$ be an element of
$\widehat{H}^0(\mathcal{L})(N^\sigma(F),\overline{\Theta})$.  Let
$\mathcal{N}_{\sigma}$ be the compact open subgroup of $N^\sigma(F)$
with
$$ \int_{\mathcal{N}_{\sigma}}\overline{\Theta}^{-1}(n)\,
\pi(n)w\,dn = 0. $$ Since $\pi$ is smooth, there exists a compact open
subgroup $\mathcal{N}_{\sigma}'\subseteq \mathcal{N}_{\sigma}$ of
finite index such that
$$ \int_{\mathcal{N}_{\sigma}}
\overline{\Theta}^{-1}(n)\,\pi(n)w\,dn = \sum_{n\in
  \mathcal{N}_{\sigma}/\mathcal{N}_{\sigma}'}
\overline{\Theta}^{-1}(n)\,\pi(n)w. $$ Let $\tilde{w}$ be a lift of
$w$ in $\mathcal{L}^\sigma$. Consider the element
$$ \tilde{w}_1 = \tilde{w} -
\frac{1}{\lvert\mathcal{N}_{\sigma}
/\mathcal{N}_{\sigma}'\rvert} \sum_{n\in \mathcal{N}_{\sigma}/\mathcal{N}_{\sigma}'}
\Theta^{-1}(n)\,\pi(n)\tilde{w}. $$
Then $\tilde{w}_1$ belongs to $\mathcal{L}(N(F),\Theta)^\sigma$ 
and $\varphi_0(\tilde{w}_1) = w$. 
Thus, we get the lemma for
$i=0$.

Now, consider the case $i=1$. Let $v$ be an element in the image of
$\varphi_1$.  Let $\tilde{v}$ be a lift of $v$ in
${\rm Ker}({\rm Nr}_{|\mathcal{L}(N(F),\Theta)})$. Then there exists
$\Gamma$ stable compact open subgroups
$\mathcal{N}'\subset \mathcal{N}$ such that $\mathcal{N}'$ fixes the
vector $\tilde{v}$, the character $\Theta$ is trivial on $\mathcal{N}'$, and
$$ \sum_{x\in X}\Theta^{-1}(x)\,\pi(x)\tilde{v} = 0, $$ 
where $X$ denotes the quotient space $\mathcal{N}/\mathcal{N}'$. Since
the $\Gamma$ action on $X\setminus X^\sigma$ is free,
there exists a subset $\mathcal{V}\subseteq X\setminus X^\sigma$ such
that $X\setminus X^\sigma
= \coprod_{i=0}^{\ell-1}\sigma^i\mathcal{V}$.
Then we have the identity
$$ \sum_{x\in X}\Theta^{-1}(X)\,\pi(x)\tilde{v}
= \sum_{x\in\mathcal{V}}\Theta^{-1}(x)\,\pi(x){\rm Nr}(\tilde{v})+
\sum_{x\in X^\sigma}\Theta^{-1}(x)\,\pi(x)\tilde{v}. $$ Thus, we get
that $\sum_{x\in X^\sigma}\overline{\Theta}^{-1}(x)\,\pi(x)v$ is
trivial in $\widehat{H}^1(\mathcal{L})$, and this implies that
$v\in \widehat{H}^1(\mathcal{L})
(N^\sigma(F),\overline{\Theta})$. Conversely, let $v$ be an element in
$\widehat{H}^1(\mathcal{L}) (N^\sigma(F),\overline{\Theta})$. Let
$\tilde{v}\in {\rm Ker}({\rm Nr})$ be a lift of $v$. Consider the
element
$$ \tilde{v}_1=\tilde{v}-\dfrac{1}
{|\mathcal{N}_\sigma/\mathcal{N}_\sigma'|} \sum_{n\in
  \mathcal{N}_\sigma/\mathcal{N}_\sigma'}
\Theta^{-1}(n)\,\pi(n)\tilde{v}. $$ Note that
$\tilde{v}_1\in \mathcal{L}(N(F),\Theta)$ and
${\rm Nr}(\tilde{v}_1)=0$ as ${\rm Nr}$ commutes with the action of
$N^\sigma(F)$.  This completes the proof.
\end{proof}
Using the above lemma, we get the following crucial result on Jacquet
modules of the Tate cohomology groups.
\begin{theorem}\label{Jacquet_tate_injective}
Let $P$ be a $\sigma$ stable $F$-parabolic subgroup of
$G$, and let $N$ be the unipotent radical of $P$. Let $\mathcal{L}$
be a smooth $\Lambda[G(F)\rtimes \Gamma]$
module. Then the natural map 
$\widehat{H}^i(\mathcal{L})_{N^\sigma(F),
\overline{\Theta}} \rightarrow
\widehat{H}^i(\mathcal{L}_{N(F),\Theta})$ is injective.  
\end{theorem}
\begin{proof}
Consider the short exact sequence of $\Lambda[\Gamma]$ modules
$$ 0\longrightarrow \mathcal{L}(N(F),\Theta)
\longrightarrow \mathcal{L}\longrightarrow
\mathcal{L}_{N(F),\Theta}\longrightarrow 0. $$ The long exact sequence
of Tate cohomology corresponding to the above short exact sequence gives
the following eact sequence:
$$ \widehat{H}^0(\mathcal{L}(N(F),\Theta))\xrightarrow
{\varphi_0}\widehat{H}^0(\mathcal{L})\rightarrow
\widehat{H}^0(\mathcal{L}_{N(F),\Theta})\rightarrow
\widehat{H}^1(\mathcal{L}(N(F),\Theta))\xrightarrow
{\varphi_1}\widehat{H}^1(\mathcal{L}) \rightarrow\widehat{H}^1
(\mathcal{L}_{N(F),\Theta}). $$ Using Lemma \ref{Tate_image}, the
image of $\varphi_i$ is equal to
$\widehat{H}^i(\mathcal{L})(N^\sigma(F),\overline{\Theta})$. Hence, we
get the theorem.
\end{proof}
As a corollary, we get the following result.
\begin{corollary}\label{Tate_cuspidal}
  Let $G$ be a connected reductive algebraic group defined over $F$,
  and let $\sigma$ be an order $\ell$ automorphism of $G$, where
  $\ell$ and $p$ are distinct primes. Let $\Pi$ be an $\ell$-modular
  cuspidal representation of $G(F)$ such that $\Pi$ extends to a
  representation of $G(F)\rtimes \Gamma$. Then each Tate cohomology
  group $\widehat{H}^i(\Pi)$ is a cuspidal representation of
  $G^\sigma(F)$.
\end{corollary}
\begin{proof}
  Let $Q$ be a proper $F$-parabolic subgroup of $G^\sigma$. By
  \cite[Lemma 12.5.4]{Prasad_Kaletha_Bruhat_Tits}, there exists a
  proper $\sigma$ stable $F$-parabolic subgroup $P$ of $G$ such that
  $P^\sigma = Q$. Let $N$ be the unipotent radical of $P$. Then
  $N^\sigma$ is equal to the unipotent radical of $Q$, to be denoted
  by $U$. Now, using Theorem \ref{Jacquet_tate_injective} for
  $\Theta=1$, the trivial character of $N(F)$, we get an injection
$$ r_{U(F)}(\widehat{H}^i(\Pi)) \hookrightarrow \widehat{H}^i(r_{N(F)}(\Pi)). $$
Since $\Pi$ is cuspidal, the Jacquet module $r_{N(F)}(\Pi)$ is trivial,
and so is $r_{U(F)}(\widehat{H}^i(\Pi))$. Thus, the Tate cohomology group
$\widehat{H}^i(\Pi)$ is cuspidal.
\end{proof}
\subsection{}
In the paper \cite[Conjecture 6.1]{Treumann_Venkatesh_linkage}, the
authors conjectured that if $\Pi$ is an irreducible $\ell$-modular
representation of $G(F)$ such that $\Pi$ extends to a representation
of $G(F)\rtimes \Gamma$, then the Tate cohomology groups
$\widehat{H}^i(\Pi)$, $i\in\{0,1\}$, are finite length representations
of $G^\sigma(F)$. In our previous work \cite[Proposition
5.5]{nadimpalli2024tate}, this conjecture is proved for $G={\rm GL}_n$
where $\sigma$ is induced by a Galois automorphism. In this
subsection, we generalise this result for arbitrary connected
reductive groups.
\subsubsection{}\label{local_bc_setup}
Let $E/F$ be a finite Galois extension with $[E:F]=\ell$, where $\ell$ and
$p$ are distinct primes. We fix a generator $\sigma$ of the cyclic
group ${\rm Gal}(E/F)$. Let $G$ be a connected reductive group defined
over $F$. Fix a minimal $F$-parabolic subgroup $P_0$ of $G$. Let
$\Omega$ denote the set of 
proper $F$-parabolic subgroups of
$G$ containing $P_0$. For $P\in \Omega$, we denote by $U$ the 
unipotent radical of $P$. 
Let $\pi$ be a smooth $R$-representation of
$G(F)$, where the coefficient field $R$ is either
$\overline{\mathbb{Q}}_\ell$ or $\overline{\mathbb{F}}_\ell$. We have the
natural map obtained from Frobenius reciprocity:
\begin{equation}\label{map}
\pi \longrightarrow \prod_{P\in \Omega}i_{P(F)}^{G(F)}\circ r_{U(F)}(\pi),
\end{equation}
and let $\pi_c$ be the kernel of the map (\ref{map}). If the Jacquet
module $r_{U(F)}(\pi_c)$ is non-zero for some $P\in \Omega$,
then Frobenius reciprocity gives the following non-zero map
\begin{equation}\label{ker_cusp_cusp}
\pi_c \longrightarrow \prod_{P\in \Omega} i_{P(F)}^{G(F)}\circ r_{U(F)}(\pi_c).
\end{equation}
Since the functor $i_{P(F)}^{G(F)}$ preserves
injective maps and $r_{U(F)}$ is an exact functor,  we get a
contradiction to the fact that $\pi_c$ is the
kernel of the map (\ref{map}). Thus, $\pi_c$ is a cuspidal
representation. Using this observation, we now prove the following
finiteness result.
\begin{theorem}\label{Tate_finite_length}
Let $E/F$ be a finite Galois extension of prime degree $\ell$, where
$\ell\ne p$. Let $G$ be a connected reductive group defined
over $F$. Let $\Pi$ be a finite length 
 $\ell$-modular representation of $[G(E)\rtimes {\rm Gal}(E/F)]$. Then, the Tate
cohomology groups $\widehat{H}^i({\rm Gal}(E/F), \Pi)$ are
finite length $\ell$-modular 
representations of $G(F)$.
\end{theorem}
\begin{proof}
We use induction on the length
of $\Pi$ and the dimension of $G$ to prove the theorem.
All Tate cohomology groups are with respect to the action of the Galois group
${\rm Gal}(E/F)$, and we drop it from the notation of Tate cohomology. 
Suppose $\Pi$ is irreducible. Consider the natural map 
\begin{equation}\label{non-cuspidal_map}
0\longrightarrow \widehat{H}^i(\Pi)_c
\longrightarrow \widehat{H}^i(\Pi) 
\longrightarrow
\prod_{P\in \Omega} i_{P(F)}^{G(F)}\circ 
r_{U(F)}(\widehat{H}^i(\Pi))
\end{equation} 
For each parabolic $P\in \Omega$
with unipotent radical $U$ and an $F$-rational Levi subgroup $M$, the 
natural map $\Pi\rightarrow r_{U(E)}(\Pi)$ gives an 
$M(F)$ equivariant injective map (Theorem \ref{Jacquet_tate_injective})
$$ r_{U(F)}(\widehat{H}^i(\Pi)) \hookrightarrow \widehat{H}^i(r_{U(E)}(\Pi)). $$
Since $r_{U(E)}(\Pi)$ is a finite length representation of $M(E)$ and
${\rm dim}(M)< {\rm dim}(G)$, the Tate cohomology
$\widehat{H}^i(r_{U(E)}(\Pi))$ has finite length as a representation
of $M(F)$ (by induction hypothesis). This implies that
$r_{U(F)}(\widehat{H}^i(\Pi))$ is a finite length $M(F)$
representation, and therefore
$i_{P(F)}^{G(F)}(r_{U(F)}(\widehat{H}^i(\Pi)))$ is a finite length
$G(F)$ representation as the parabolic induction preserves finite
length (\cite[Chapter II, Section 5.13]{Vigneras_modl_book}). Thus,
the representation
$$ \prod_{P\in \Omega} i_{P(F)}^{G(F)}\circ r_{U(F)}(\widehat{H}^i(\Pi))$$
is of finite length. Note that the center $Z(E)$ of $G(E)$ acts on the
space of $\Pi$ by a character (as $\Pi$ is irreducible). We deduce
that the center $Z(F)$ of $G(F)$ acts on $\widehat{H}^i(\Pi)$ by a
character. Hence, the cuspidal $G(F)$ representation
$\widehat{H}^i(\Pi)_c$, being a subrepresentation of
$\widehat{H}^i(\Pi)$, admits a central character. Since
$\widehat{H}^i(\Pi)$ is finitely generated as a representation of
$G(F)$ (Theorem \ref{Tate_finite_gen}), it follows from
\cite[Corollary 4.5]{Dat_finiteness}) that $\widehat{H}^i(\Pi)_c$ is
also finitely generated. Thus, $\widehat{H}^i(\Pi)_c$ is a finitely
generated $\ell$-modular cuspidal representation of $G(F)$ with a
central character and we get that $\widehat{H}^i(\Pi)_c$ is of finite
length (\cite[Section 7.4, Chapter I]{Vigneras_modl_book}). We
conclude from (\ref{non-cuspidal_map}) that the $\ell$-modular $G(F)$
representation $\widehat{H}^i(\Pi)$ is of finite length.

Now, consider the general case where $\Pi$ is not necessarily an
irreducible representation of $\overline{\mathbb{F}}_\ell[G(F)]$.
Let $\sigma$ be a generator of $\Gamma={\rm Gal}(E/F)$. 
Let
$\Pi_1$ be an irreducible sub-representation of $\Pi$ and set 
$\Pi_1'$ to be the sub-representation of $\Pi$ spanned by
$\Pi_1^{\sigma^i}$, for $0\leq i<\ell$. Note that $\Pi_1'$ is
$\sigma$-stable, and we consider the short exact
sequence of
$\overline{\mathbb{F}}_\ell[G(E)\rtimes \Gamma]$ modules
$$ 0\longrightarrow \Pi_1'\longrightarrow \Pi\longrightarrow \Pi/
\Pi_1'\longrightarrow 0. $$
From the long exact sequence of Tate cohomology, we get the following
exact sequence:
$$\widehat{H}^i(\Pi_1')\rightarrow
\widehat{H}^i(\Pi)\rightarrow \widehat{H}^i(\Pi/\Pi_1').$$
If $\Pi_1$ is not $\sigma$-stable, then $\widehat{H}^i(\Pi_1')$ is
trivial for $i\in \{0,1\}$, and if $\Pi_1$ is $\sigma$-stable, then
$\widehat{H}^i(\Pi_1')$ is a finite length representation of $G(F)$
from the previous case.  
Note that the length of $\Pi/\Pi_1'$ is now strictly less than that of
$\Pi$. Then, using induction hypothesis, the Tate cohomology
$\widehat{H}^i(\Pi/\Pi_1')$ is a finite length
representation of $G(F)$. Thus, we get that $\widehat{H}^i(\Pi)$ is a
finite length representation of $G(F)$.
\end{proof}
As a corollary, we get the following result.
\begin{corollary}
  Let $E/F$ be a finite Galois extension of prime degree $\ell$, where
  $\ell\ne p$. Let $G$ be a connected reductive group defined over
  $F$. Let $\Pi$ be a finite length integral $\mathcal{K}$
  representation of $G(E)\rtimes {\rm Gal}(E/F)$. Then, for any
  $\Lambda[G(E)\rtimes {\rm Gal}(E/F)]$ stable lattice $\mathcal{L}$
  in $\Pi$, the Tate cohomology spaces
  $\widehat{H}^i({\rm Gal}(E/F), \mathcal{L})$ are finite length
  $\ell$-modular representations of $G(F)$.
\end{corollary}
\begin{proof}
Recall that $\mathcal{L}$ is a free $\Lambda$-module and that
$\mathcal{L}/\ell\mathcal{L}$ is a finite length representation of $G(E)$
(see \cite[II.5.11.a]{Vigneras_modl_book} for finiteness). Consider the short
exact sequence of $G(E)\rtimes{\rm Gal}(E/F)$ modules
$$ 0\longrightarrow \mathcal{L}
\xrightarrow{{\rm mult.}\ell}
\mathcal{L}\longrightarrow \mathcal{L}/\ell\mathcal{L}
\longrightarrow 0. $$ 
By Theorem \ref{Tate_finite_length}, the Tate cohomology groups
$\widehat{H}^i(\mathcal{L}/\ell\mathcal{L})$ are of finite length as
representations of $G(F)$. From the long exact
sequence of Tate cohomology corresponding to the above short exact
sequence, we have
$$ 0\rightarrow \widehat{H}^0(\mathcal{L})\rightarrow
\widehat{H}^0(\mathcal{L}/\ell\mathcal{L})
\rightarrow
\widehat{H}^1(\mathcal{L})\rightarrow 0. $$
and 
$$ 0\rightarrow \widehat{H}^1(\mathcal{L})\rightarrow
\widehat{H}^1(\mathcal{L}/\ell\mathcal{L})\rightarrow
\widehat{H}^0(\mathcal{L})\rightarrow 0. $$
Thus, we get that each
$\widehat{H}^i(\mathcal{L})$ is of finite length.
\end{proof}
\section{Some genericity and irreducibility results}\label{sec_gen_irr}
\label{Tate_generic_subsection}
In this section, $G$ will be a
split connected reductive algebraic group defined over $F$. We follow the
notations defined in Subsection \ref{local_bc_setup}.
\subsubsection{}
For this section, we follow \cite{Derived_subgroups_Henniart_Bushnell}. 
Let $B$ be a Borel subgroup of $G$ defined over $F$.
Let $E$ be a cyclic Galois extension of $F$ with $[E:F]=\ell$.
Let $S$ be a maximal $F$-split torus and let $\Phi$ the set of roots
of $S$ in  $G$, and let $\Delta$ be the set of simple roots in
$\Phi$ for the choice of $B$. Let $\mathfrak{g}$ be the Lie algebra of
$G$ and let $\mathfrak{g}_\alpha$ be the subspace of $\mathfrak{g}$
for which $S$ acts via $\alpha$. The root subgroup corresponding to
$\alpha\in \Delta$ is denoted by $U_{\alpha}$, this is the unique
connected subgroup of $G$ whose Lie algebra is 
$\mathfrak{g}_\alpha$. Let $R$ denote $\overline{\mathbb{Q}}_\ell$ or
$\overline{\mathbb{F}}_\ell$. A character $\Theta :U(F)\rightarrow
R^\ast$ is called non-degenerate whenever $\Theta$ is non-trivial
on $U_{\alpha}(F)$ for all $\alpha\in \Delta$. Note that
$$U(E)/U(E)_{\rm der}\simeq \prod_{\alpha\in
  \Delta}V_\alpha(E),$$
  where $V_\alpha$ is the one-dimensional $F$-vector group.  Thus we choose a
non-degenerate character
$\Theta_E:U(E)\rightarrow R^\ast$,
which is ${\rm Gal}(E/F)$ stable and the restriction of $\Theta_E$ to
$U(F)$ is denoted by $\Theta_F$ (which is a non-degenerate
character of $U(F)$). We set
$$\mathfrak{m}_F=(U(F),\Theta_F)\ \text{and}\  \mathfrak{m}_E=(U(E),
\Theta_E).$$ 
The multiplicity one theorem says that
$$ {\rm dim}_R \big({\rm Hom}_{G(F)}(V,\,{\rm Ind}_{U(F)}^{G(F)}\Theta_F)\big)\leq 1. $$
An irreducible $R$-representation $(\pi,V)$ of $G(F)$ is called
$\mathfrak{m}_F$-generic if
$$ 
{\rm dim}_R\big({\rm Hom}_{G(F)}(V,\,{\rm
  Ind}_{U(F)}^{G(F)}\Theta_F)\big) = 1.
$$ 
An equivalent condition for $(\pi,V)$ to be
$\mathfrak{m}_F$-generic--which we will use frequently--is that the
twisted Jacquet module $V_{U(F),\Theta_F}\simeq R$. We also have
the obvious notions of $\mathfrak{m}_E$-generic representations of
$G(E)$.  Let $(\pi,V)$ be a smooth generic $R$-representation of
$G(F)$.  Let $\mathcal{J}_\pi$ be a non-zero linear functional in
${\rm Hom}_{U(F)}(V,\Theta_F)$. Let
$\mathbb{W}(\pi,\Theta_F) \subset {\rm Ind}_{U(F)}^{G(F)}(\Theta_F)$
be the space of functions $W_v:G(F)\rightarrow R$, $v \in V$, where
$$ W_v(g) := \mathcal{J}_\pi
\big(\pi(g)v\big), $$
for $g \in G(F)$. The map $v\mapsto W_v$
induces a $G(F)$
equivariant isomorphism from
$(\pi,V)$ onto $\mathbb{W}(\pi,\Theta_F)$.
The space 
$\mathbb{W}(\pi,\Theta_F)$ is
called the {\it Whittaker model} of
$(\pi,V)$.
\subsubsection{}
We now prove that the ${\rm Gal}(E/F)$-Tate cohomology of a generic
$\ell$-modular ${\rm Gal}(E/F)$ stable representation of $G(E)$ has a
unique generic sub-quotient. We proved this result for the case where
$G={\rm GL}_n/F$ using Kirillov model.  However, we give a simpler
proof here. Which will be used in the later sections to prove
irreducibility of Tate cohomology groups of certain representations of
${\rm GL}_n$ obtained as Base change lifts of cuspidal
representations.
\begin{theorem}\label{Tate_generic} 
Let $E/F$ be a finite Galois extension with $[E:F]=\ell$. Let $G$ be
a split reductive algebraic group defined over $F$. Let
$(\Pi,V)$ be a ${\rm Gal}(E/F)$ stable $\mathfrak{m}_E$-generic
$\ell$-modular representation of $G(E)$.  Then the Tate cohomology
group $\widehat{H}^0({\rm Gal}(E/F), \Pi)$ admits a unique
$\mathfrak{m}_F$-generic sub-quotient.
\end{theorem}
\begin{proof}
Let $\Gamma$ be the Galois group ${\rm Gal}(E/F)$.
Let $\mathcal{J}$ be a non-zero element in
${\rm Hom}_{U(E)}(\Pi,\Theta_E)$.  Let $\mathcal{W}$ be the
Whittaker model of $\Pi$. Using Proposition
\ref{Jacquet_tate_injective}, the Tate cohomology applied to the surjective map
$\mathcal{W}\rightarrow \mathcal{W}_{U(E), \Theta_E}$ gives the
following injection:
$$ \widehat{H}^0(\mathcal{W})_{U(F),\Theta_F}
\hookrightarrow \widehat{H}^0(\mathcal{W} _{U(E),\Theta_E}).$$ 
Since $ {\rm Gal}(E/F)$
acts trivially on the twisted Jacquet module
$\mathcal{W}_{U(E),\Theta_E}\simeq \overline{\mathbb{F}}_\ell$, the
Tate cohomology group $\widehat{H}^0(\mathcal{W}_{U(E), \Theta_E})$ is
a one dimensional vector space over $\overline
{\mathbb{F}}_\ell$. Now, let us consider the restriction to $G(F)$ map
$$ \Phi : \mathcal{W}^\Gamma\longrightarrow {\rm Ind}_{U(F)}^{G(F)}\Theta_F$$
$$ W_v \longmapsto {\rm res}_{G(F)}W_v. $$ 
It is enough to show that $\Phi$ is non-zero.  Suppose
$\Phi(W_v) =0$ for all $W_v\in\mathcal{W}^\Gamma$. Let
$v\in {\rm Ker}(\mathcal{J})\cap \mathcal{W}^\Gamma$.  Note that
$\Pi(g)v\in \Ker(\mathcal{J})$, for $g\in G(F)$ and $g\in U(E)$.  This
shows that $\Pi(g)v\in {\rm Ker}(\mathcal{J})$, for all
$g\in G^{\rm der}(E)$. Also note that ${\rm Ker}(\mathcal{J})$ is
stable under the center $Z(E)$, and hence $\langle G(E).v\rangle$ is
contained in ${\rm Ker}(\mathcal{J})$. Now, by the irreducibility of
$\Pi$, we get a contradiction to the assumption that the map $\Phi$ is
non-zero.
\end{proof}
\subsection{Irreducibility of 
Tate cohomology for ${\rm GL}_n$}
Let $G$ be the group ${\rm GL}_n$, and we prove the irreducibility of
the ${\rm Gal}(E/F)$-Tate cohomology of
${\rm GL}_n(E)\rtimes {\rm Gal}(E/F)$ stable lattices in certain integral
generic $\mathcal{K}$-representations of ${\rm GL}_n(E)$. Assume
that $n=\ell m$ for some integer $m\geq 1$.  Let $(\Pi,V)$ be the irreducible
representation
\begin{equation}\label{bc_noncuspidal}
\tau\times\tau^\sigma\times\cdots\times
\tau^{\sigma^{\ell-1}},
\end{equation} 
where $\tau$ is an absolutely irreducible cuspidal
$\mathcal{K}$-representation of ${\rm GL}_m(E)$ such that $\tau$ is
not isomorphic to twisted representation $\tau^\sigma$. Note that
$\Pi\simeq \Pi^\sigma$, which induces an action of ${\rm Gal}(E/F)$ on
the space $V$ that is compatible with the ${\rm Gal}(E/F)$ action on
${\rm GL}_n(E)$. We further assume that $\tau$ is integral, i.e.,
there exists a ${\rm GL}_m(E)$ stable $\Lambda$-lattice in the
underlying space of $\tau$. The representation $\Pi$ admits a
${\rm GL}_n(E)\rtimes {\rm Gal}(E/F)$ stable $\Lambda$-lattice in $V$.
We now prove the following result.
\begin{theorem}\label{Tate_bc_cuspidal}
  Let $(\Pi,V)$ be an absolutely irreducible, integral, generic
  $\mathcal{K}$-representation of ${\rm GL}_n(E)$ of the form
  (\ref{bc_noncuspidal}). Then for any
  ${\rm GL}_n(E)\rtimes {\rm Gal}(E/F)$ stable $\Lambda$-lattice
  $\mathcal{L}\subseteq V$, the zeroth Tate cohomology group
  $\widehat{H}^0({\rm Gal}(E/F),\mathcal{L})$ is irreducible and
  cuspidal. Moreover, the first Tate cohomology group
  $\widehat{H}^1({\rm Gal}(E/F),\mathcal{L})$ is trivial.
\end{theorem}
\begin{proof}
  We set $\Gamma = {\rm Gal}(E/F)$. To see the cuspidality of the Tate
  cohomology groups $\widehat{H}^i(\Gamma,\mathcal{L})$, consider the
  Jacquet module $\widehat{H}^i(\Gamma,\mathcal{L})_{N(F)}$, where $N$
  is the unipotent radical of an $F$-rational proper parabolic
  subgroup $P$ of ${\rm GL}_n$. Let $M$ be an $F$-rational Levi
  subgroup of $P$. Using Proposition \ref{Jacquet_tate_injective}, we
  have $M(F)$ equivariant injection
$$ \widehat{H}^i(\Gamma, \mathcal{L})_{N(F)} 
\hookrightarrow \widehat{H}^i(\Gamma, \mathcal{L}_{N(E)}). $$ We show
that the Tate cohomology group
$\widehat{H}^i(\Gamma, \mathcal{L}_{N(E)})$ is trivial for
$i\in \{0,1\}$. Say $W$ is a finite length integral $\ell$-adic
$M(E)\rtimes \Gamma$ representation such that the Jordan--Holder
factors of $W$, as a representation of $M(E)$, are multiplicity
free. Assume that all irreducible $M(E)$ sub-representation of $W$ are
not stable under the action of $\Gamma$. The Jacquet module $V_{N(E)}$
is a special case of $W$ and it is integral as $\mathcal{L}_{N(E)}$ is
a $M(E)\rtimes \Gamma$ stable $\Lambda$-lattice in $V_{N(E)}$ by
\cite[Proposition 1.4] {Dat_tempered}.  We use induction on the length
of $W$ to show that $\widehat{H}^i(\Gamma, \mathcal{W})$ is trivial,
for all $M(E)\rtimes \Gamma$ stable lattice $\mathcal{W}$ in $W$.

Let $W_1$ be an irreducible $M(E)$ sub-representation of $W$
and assume that $W_1$ is not fixed by $\Gamma$. 
The $\Lambda$-module
$$ \mathcal{M}:=\bigoplus_{i=0}
^{\ell-1}\sigma^i (\mathcal{W}\cap W_1) $$ is a $M(E)\rtimes \Gamma$
stable lattice in $W_2=\bigoplus_{i=0}^{\ell-1} \sigma^i(W_1)$ and it
is contained in the lattice $\mathcal{W}$. Since $\mathcal{M}$ is a
projective $\Lambda[\Gamma]$-module, the Tate cohomology group
$\widehat{H}^i (\Gamma, \mathcal{M})$ is trivial for $i\in
\{0,1\}$. Now, using the long exact sequence of Tate cohomology
corresponding to the following short exact sequence
$$ 0\longrightarrow \mathcal{M}\longrightarrow
\mathcal{W}\longrightarrow \mathcal{W}/\mathcal{M}
\longrightarrow 0, $$ 
we get that
$$\widehat{H}^i(\Gamma, \mathcal{W})\simeq
\widehat{H}^i(\Gamma, \mathcal{W}
/\mathcal{M}). $$ Since the quotient
$\mathcal{W}/\mathcal{M}$ is a
$M(E)\rtimes \Gamma$ stable $\Lambda$-lattice in
$W/W_2$, and the length of $W/W_2$ is
strictly less than that of $W$, the Tate cohomology group
$\widehat{H}^i(\Gamma, \mathcal{W} /\mathcal{M})$ is trivial (by
induction). Therefore, $\widehat{H}^i(\Gamma, \mathcal{W})$ is
trivial. 
Thus, we get that the Tate cohomology 
$\widehat{H}^i(\Gamma, \mathcal{L}_{N(E)})$ is trivial. 
Hence, the representation $\widehat{H}^i(\Gamma, \mathcal{L})$ is cuspidal.
The $G(F)$ representation 
$\widehat{H}^0(\Gamma, \mathcal{L})$ admits a unique
generic sub--quotient (Theorem \ref{Tate_generic}).  Since
$\widehat{H}^0(\Gamma, \mathcal{L})$ is also cuspidal, it is
irreducible. From the inclusion 
$$ \widehat{H}^1(\Gamma, \mathcal{L})_{N(F),\Theta_F}
\hookrightarrow \widehat{H}^1 (\Gamma, \mathcal{L}_{N(E),\Theta_E}). $$
we get that
$\widehat{H}^1(\Gamma, \mathcal{L}_{N(E),\Theta_E})$ is trivial, and so is
$\widehat{H}^1(\Gamma, \mathcal{L})_{N(F), \Theta_F}$. Since
$\widehat{H}^1(\Gamma, \mathcal{L})$ is also cuspidal, we get that
$\widehat{H}^1(\Gamma, \mathcal{L})$ is trivial.
\end{proof}
\begin{remark}\label{Remark_Tate_bc_cuspidal}
\normalfont If $\Pi$ is an absolutely irreducible, integral, generic
$\mathcal{K}$-representation of the form (\ref{bc_noncuspidal}), and
the mod-$\ell$ reduction $r_\ell(\Pi)$ is irreducible, then it
follows from a similar argument as in Theorem
\ref{Tate_bc_cuspidal} that 
$\widehat{H}^0({\rm Gal}(E/F), r_\ell(\Pi))$ is an
irreducible $\ell$-modular cuspidal representation of ${\rm GL}_n(F)$.
\end{remark}
\section{Tate cohomology and Shintani correspondence}\label{sec_shintani}
In this section, we study ${\rm
  Gal}(\mathbb{F}_{q^\ell}/\mathbb{F}_{q})$-Tate cohomology of Galois
invariant lattices in Shintani liftings of cuspidal representations of
${\rm GL}_n(\mathbb{F}_q)$ to ${\rm GL}_n(\mathbb{F}_{q^\ell})$. The
primary novelty of this section is that it treats the case where the
Shintani lifting of a cuspidal representation is not necessarily
cuspidal. 
\subsection{Character formula for finite groups of Lie type}
\subsubsection{}
We begin with the character theory of the finite group
${\rm GL}_n(\mathbb{F}_q)$. Although, irreducible characters of
${\rm GL}_n(\mathbb{F}_q)$ are obtained by J. A. Green
(\cite{Green_finite_groups}), for the present purpose, we find it
convenient to use Deligne--Lusztig theory. Let us consider a connected
reductive algebraic group $G$ defined over $\mathbb{F}_q$, and let
$F:G\rightarrow G$ be the corresponding Frobenius endomorphism. Let
$(S_0, B_0)$ be a pair consisting of an $F$-stable Borel subgroup
$B_0$ and an $F$-stable maximal torus $S_0$ of $B_0$. The automorphism
induced by the Frobenius endomorphism $F$ on the Weyl group
$W(G, S_0)$ is again denoted by $F$. For an $F$-stable maximal torus
$T$ of $G$, there exists $w\in W(G, S_0)$ such that $(T, F)$ is
isomorphic to $(S_0, {\rm Ad}(w)\circ F)$. Two $F$-stable maximal tori
$T_1$ and $T_2$ of $G$ are conjugate in $G^F$ if and only if the
corresponding Weyl group elements $w_1$ and $w_2$ are $F$-conjugate in
$W(G, S_0)$, i.e., if $w_1=ww_2F(w)^{-1}$, for some $w\in W(G, S_0)$.
For $w\in W(G, S_0)$, we denote by $T_w$ an $F$-stable maximal torus
of $G$ such that $(T_w, F)$ is isomorphic to
$(S_0, {\rm Ad}(w)\circ F)$. For the above results, we refer to
\cite[Corollary 1.14]{Deligne_Luszting_paper}.
\subsubsection{}
For an $F$-stable maximal torus $T$ and a character $\theta$ of $T^F$,
Deligne--Lusztig associate a virtual representation $R^G_{T,
  \theta}$. Every irreducible $\ell$-adic representation of $G^F$
occurs in a virtual character $R^G_{T, \theta}$, for some torus $T$
and $\theta\in \widehat{T^F}$.  For all unipotent elements $u\in G^F$,
we denote by $Q_{G, T, F}(u)$, the Green's function, which is defined
as the value $R^G_{T, 1}(u)$.  Let $g\in G^F$ be an element with
Jordan decomposition $g=su$. Let $C_G(s)^0$ be the connected component
of the centralizer of $s\in G^F$. Let $T_1,T_2,\dots, T_k$ be
representatives of the $(C_G(s)^0)^F$-conjugacy classes of $F$-stable
maximal tori in $C_G(s)^0$, which are conjugate to $T$ in $G^F$. For
simplicity, we denote by $H$ the group $C_G(s)^0$.  
The Deligne--Lusztig character formula 
(see \cite[Theorem 4.2]{Deligne_Luszting_paper}
or \cite[Lemma 2.2.23, Chapter 2]{Malle_Geck_characters})
gives
\begin{equation}\label{DL_char_form_1}
R_{T, \theta}^G(su)=\sum_{1\leq i\leq k}Q_{H, T_i, F}(u)\sum_{w\in W(G,
T_i)^F/W(H, T_i)^F}\theta(wsw^{-1}).
\end{equation}
Let $T$ be an $F$-stable maximal torus of $G$ such that $T$ is
contained in an $F$-stable parabolic subgroup $P$ of $G$. Let $M$ be
an $F$-stable Levi-subgroup of $P$. Let $\theta\in
\widehat{T^F}$. Then we have
\begin{equation}
R_{T, \theta}^G={\rm Ind}_{P, M}^G
(R^M_{T, \theta}). 
\end{equation}
The notation ${\rm Ind}_{P,M}^G$ means the parabolic induction of a
virtual character of $M$ to $G$. 
\subsubsection{}
When $G$ is ${\rm GL}_n$ and $F$ is the Frobenius
endomorphism induced by
$x\mapsto x^q$, the action of $F$ on $W(G, S_0)\simeq S_n$ is trivial
and the $F$-stable maximal tori of $G$ are parametrised by the conjugacy classes of
$S_n$.  Every cuspidal representation of $G^F$ is equal to
$(-1)^{n-1}R_{T, \theta}^G$, where $T$ is the torus corresponding to the
conjugacy class of $W(G, S_0)$ containing $n$-cycles, and
$\theta\in \widehat{T^F}$ is in general position, i.e.,
$\theta^w\neq \theta$ for all $w\in W(G, T)^F$. For an unipotent
element $u\in G^F$ with associated partition $\mu$ of the integer
$n$, and a maximal torus corresponding to a partition $\lambda$ of the
integer $n$, the Green's function $Q_{G, T, F}(u)$ is explicitly
computed by J. A. Green, and they are related to Hall polynomials. We
will not make use of this relation with Hall polynomials, however, we
will use the connection of Green's polynomials with the representation
theory of $S_n$ due to T. Springer. 
In the fundamental work (\cite[Theorem 7.8]{TA_Springer}), 
T. Springer showed that
$$Q_{G, T_w, F}(u)=\sum_{i\geq 0}a_i(w, u)q^i$$
where the map
$$w\mapsto a_i(w, u)$$
is the character of an irreducible representation of $S_n$. 
These results have their analogues for all finite reductive groups. 
\subsubsection{}
Let $F:{\rm GL}_n\rightarrow {\rm GL}_n$ be the Frobenius endomorphism
of ${\rm GL}_n$ induced by the map $x\mapsto x^q$ on the entries of
the matrices. Let $w\in S_n$ be an $n$-cycle, and let $T_w$ be an
$F$-stable maximal torus of ${\rm GL}_n$ corresponding to the element
$w$. Let $s\in T_w^F$ be an element such that the characteristic
polynomial of $s$ is $f(X)^r$ where $f(X)\in \mathbb{F}_q[X]$ is an
irreducible polynomial of degree $d$. We have $n=dr$. Let $H$ be the
connected group $C_{{\rm GL}_n}(s)$.  Note that
$H(\overline{\mathbb{F}}_q)$ is isomorphic to
${\rm GL}_r(\overline{\mathbb{F}}_q)^d$ and the Frobenius endomorphism
$F:H\rightarrow H$ is given by
$$(g_1,g_2,\dots, g_d)\longmapsto ({\rm Fr}(g_d), {\rm Fr}(g_1),\dots,
{\rm Fr}(g_{d-1})),$$
where ${\rm Fr}$ is the Frobenius endomorphism of ${\rm GL}_r$
sending a matrix $(g_{ij})\in {\rm GL}_n(\overline{\mathbb{F}}_q)$
to $(g_{ij}^{q})$. Then $H^F$ is
isomorphic to ${\rm GL}_r(\mathbb{F}_{q^d})$. Let $D_r$ and $B_r$ be
the diagonal torus and the upper triangular Borel subgroup of
${\rm GL}_r$. Note that $D_r^d$ and $B_r^d$ give an $F$-stable maximal
torus and an $F$-stable Borel subgroup to be denoted by $S_H$ and
$B_H$, respectively. The Frobenius endomorphism $F$ on $H$ induces a
map $F^\ast$ on $W(H, S_H)$. If we identify $W(H, S_H)$ with
$S_{r}^d$, then
$$F^\ast(\sigma_1,\sigma_2,\dots, \sigma_d)=(\sigma_d, \sigma_1,\dots,
\sigma_{d-1}).$$
The next lemma is about the $F^\ell$-stable maximal tori in $H$ which are
conjugate to $T_w$ by an element in ${\rm
GL}_n(\mathbb{F}_{q^\ell})$. Along the way, we prove some standard
observations which we could not find a convenient reference.
\begin{lemma}\label{Conjugacy_maximal_torus}
Let $w$ be an $n$-cycle in $S_n$, and let $T_w$ be an $F$-stable
maximal torus of ${\rm GL}_n$ as above.
Let $s\in T_w^F$ be an element with
characteristic polynomial $f(X)^r$, where $f(X)\in \mathbb{F}_q[X]$
is an irreducible polynomial of degree $d$. 
\begin{enumerate}
\item Let $S_1$ and $S_2$ be two $F$-stable maximal tori of $H$ such that
$gS_1g^{-1}=S_2$ for some $g\in {\rm GL}_n(\mathbb{F}_q)$. Then
there exists an element $h\in H^F$ such that $hS_1h^{-1}=S_2$. 
\item Assume that $(\ell, d)=1$. Let $S_1$ and $S_2$ be two $F^\ell$-stable
maximal tori of $H$ such that $gS_1g^{-1}=S_2$ for some
$g\in {\rm GL}_n(\mathbb{F}_{q^\ell})$. Then there exists an element
$h\in H^{F^\ell}$ such that $hS_1h^{-1}=S_2$.
\item Let $S$ be an $F$-stable maximal
torus of $H$ such that $gSg^{-1}=T_w$ for some $g\in {\rm
GL}_n(\mathbb{F}_{q^\ell})$. There exists an element $h\in H^{F^\ell}$ such
that $hSh^{-1}=T_w$. 
\end{enumerate}
\end{lemma}
\begin{proof}
Let $w_1$ and $w_2$ be two elements of $W(H, S_H)$ corresponding to $S_1$ and $S_2$,
respectively. We
identify the elements $w_1$ and $w_2$ by the $d$-tuples
$(s_1,\dots,s_d)$ and $(t_1,\dots,t_d)$ respectively, via the
isomorphism
$$ W(H,S_H)\simeq S_{r}\times\cdots\times S_{r},$$
where $s_i, t_j \in S_{r}$ for $1\leq i,j\leq d$. If there exists an element
$k\in S_{r}$  such that
$$ks_1s_2\cdots s_dk^{-1}=t_1t_2\cdots t_d.$$
We define $k_r$ by setting
$$ k_r = (t_r\cdots t_2t_1)\, k\,(s_r\cdots s_2s_1)^{-1},\ 1\leq r\leq d-1. $$
Let $w=(k_1,\dots,k_{d-1},k)$ be an element of $W(H,S_H)$, and we have
$$ww_1F^*(w)^{-1} = w_2.$$ Thus, $w_1$ is $F^*$-conjugate to $w_2$.
This shows that the tuples $(s_1,\dots,s_d)$ and $(t_1,\dots,t_d)$ are
$F^\ast$ conjugate in $W(H, S_H)$ if and only if $s_1s_2\cdots s_d$ and
$t_1t_2\cdots t_d$ are conjugate in $S_r$. We deduce that any
$F^\ast$-conjugacy class in $W(H, S_H)$ has a representative
$(s_1,\dots, s_d)$ such that $s_1=s_2=\cdots=s_{d-1}={\rm id}$ and we
choose such a representatives for the tuples $(s_1, s_2,\dots, s_d)$ and
$(t_1,t_2,\dots, t_d)$. Assume that $s_d$ corresponds to the partition
$(r_1,r_2,\dots, r_k)$ of $r$. Then $S_1$ is an
$F$-stable maximal torus of ${\rm GL}_n$ corresponding to the partition
$(dr_1,dr_2,\dots, dr_k)$. This shows the lemma in this case.
	
The part $(2)$ of the lemma is similar to part $(1)$. First assume that
$(\ell, d)=1$. Note that the Frobenius endomorphism $F$ on $H$ is the
composition $\omega\circ {\rm Fr}$, where $\omega$ is the $d$-cycle
$(1,2,\dots,d)$. Then we have the identity
$F^{\ell} = \omega^\ell\circ {\rm Fr}^\ell$. Since $\ell$ does not
divide $d$, $w^\ell$ is also a $d$-cycle. Using a similar argument as
in part $(1)$, we get that $S_1$ and $S_2$ are conjugate in
$H^{F^{\ell}}$.

For part $(3)$ of the lemma, using the above discussion for every
$(F^\ast)^\ell$-conjugacy class in $W(H, S_H)$, we can associate
$\ell$-partitions $(r_{i1}, r_{i2},\dots, r_{ik_i})$ of the positive integer 
$d/\ell$, for $1\leq i\leq \ell$; and the torus associated with this 
$(F^\ast)^\ell$-conjugacy class of
$W(H, S_H)$ is ${\rm GL}_n(\mathbb{F}_{q^\ell})$-conjugate to a torus
of ${\rm GL}_n$ corresponding to the partition
$$(dr_{11}/\ell, dr_{12}/\ell,\dots, dr_{1k_1}/\ell, dr_{21}/\ell,
dr_{22}/\ell,\dots, dr_{2k_2}/\ell, \dots, dr_{\ell1}/\ell,
dr_{\ell2}/\ell,\dots, dr_{lk_\ell}/\ell).$$ Now the lemma follows
since $k_i=1$, for $1\leq i\leq \ell$ and $r_{i1}=r$.  This completes
the proof.
\end{proof}
\subsection{Shintani correspondence}\label{BC_finite_groups}
Let $\sigma$ be the Frobenius element in
${\rm Gal}(\mathbb{F}_{q^\ell}/\mathbb{F}_q)$.  Let $\Pi$ be an $\ell$-adic irreducible
representation of ${\rm GL}_n(\mathbb{F}_{q^\ell})$ such that $\Pi$ is
${\rm Gal}(\mathbb{F}_{q^\ell}/\mathbb{F}_q)$
equivariant, i.e., $\Pi\simeq \Pi^\sigma$.  In the article \cite[Theorem
1]{Shintani_BC}, Shintani proved that for a suitable choice of
isomorphism $I_\sigma:\Pi\rightarrow \Pi^\sigma$, there exists an $\ell$-adic
irreducible representation $\pi$ of ${\rm GL}_n(\mathbb{F}_q)$ such that
\begin{equation}\label{BC_Char_formula}
{\rm Tr}(I_\sigma\circ \Pi(g)) = 
{\rm Tr}\big(\pi({\rm Norm}_{\mathbb{F}_{q^\ell}/\mathbb{F}_q}(g))\big),\ 
\forall g\in {\rm GL}_n(\mathbb{F}_{q^\ell}),
\end{equation}
where ${\rm Norm}_{\mathbb{F}_{q^\ell}/\mathbb{F}_q}(g)$ is the unique
conjugacy class of ${\rm GL}_n(\mathbb{F}_q)$ containing the element
$$g\sigma(g)\cdots\sigma^{\ell-1}(g).$$ The association $\Pi\mapsto \pi$
gives a bijection from the set of all isomorphism classes of
${\rm Gal}(\mathbb{F}_{q^\ell}/\mathbb{F}_q)$ equivariant irreducible $\ell$-adic
representation of ${\rm GL}_n(\mathbb{F}_{q^\ell})$ onto the set of
isomorphism classes of irreducible $\ell$-adic representations of
${\rm GL}_n(\mathbb{F}_q)$. The inverse of this bijection is called the
Shintani base change lifting, denoted by
${\rm bc}_{\mathbb{F}_{q^\ell}/\mathbb{F}_q}$.
\subsection{Congruence of characters}
In this subsection, we prove certain congruences between the character
of a cuspidal representation and the character of its Shintani lifting
to ${\rm GL}_n(\mathbb{F}_{q^\ell})$, when evaluated at an element of
${\rm GL}_n(\mathbb{F}_q)$. We follow the notations of the preceding
subsection \ref{BC_finite_groups}, and let $F$ be the Frobenius
endomorphism of ${\rm GL}_n/\mathbb{F}_q$ induced by the Frobenius
automorphism
$\overline{\mathbb{F}}_q\rightarrow
\overline{\mathbb{F}}_q$,\,\,$x\mapsto x^q$. We begin with the
following preliminary lemma.  For any representation $\tau$ of
${\rm GL}_n(\mathbb{F}_{q^\ell})$, we denote by $\tau^F$ the
composition of the representation $\tau$ with the automorphism of
${\rm GL}_n(\mathbb{F}_{q^\ell})$ induced by $\sigma$. For any representation
$\rho$, we denote by $\chi_\rho$ its character. 
\begin{lemma}\label{char_iden_2}
  Let $\tau$ be an irreducible $\ell$-adic representation of
  ${\rm GL}_n(\mathbb{F}_{q^\ell})$. Let $\sigma$ be the Frobenius
  element in ${\rm Gal}(\mathbb{F}_{q^\ell}/\mathbb{F}_q)$, and let
  $F$ be the Frobenius endomorphism of ${\rm GL}_n/\mathbb{F}_q$.  Let
\begin{align*}
\eta=&\tau\otimes\tau^F\otimes \cdots\otimes
\tau^{F^{\ell-1}},\\
\iota(g)=& (g, g^F,\dots, g^{F^{\ell-1}}).
\end{align*} 
Then, we have
$$ \sum_{w\in S_l} 
\chi^{(w)}_\eta(\iota(g))
\equiv \sum_{i=0}^{\ell-1}
\chi_\tau(g^{F^i})^\ell\ \text{mod}\ (\ell),$$ for all
$g\in {\rm GL}_m(\mathbb{F}_{q^\ell})$, where $S_\ell$ is permutation group
of $\ell$ symbols.
\end{lemma}
\begin{proof}
For any $w\in S_\ell$, we have the
following character identity
$$ \chi_\eta^{(w)}(g_0,\dots,g_{\ell-1}) = 
\chi_\eta\,(g_{w^{-1}(0)},
\dots,g_{w^{-1}(\ell-1)}) = 
\prod_{i=0}^{\ell-1}\chi_\tau^{F^i}
(g_{w^{-1}(i)}), $$
for all $g_i\in {\rm GL}_m(\mathbb{F}_{q^\ell})$. 
Then
\begin{equation}\label{char_formula_2}
\sum_{w\in S_\ell} 
\chi^{(w)}_\eta(\iota(g)) = \sum_{w\in S_\ell} 
\Big(\prod_{i=0}^{\ell-1}\,
\chi_\tau(g^{F^{i+w(i)}})\Big).
\end{equation}
Let $X$ denote the set of functions
$$\{{\rm id}+w:\mathbb{Z}/\ell\mathbb{Z}\rightarrow 
\mathbb{Z}/\ell\mathbb{Z}: w\in S_\ell\},$$ and let $w_0$ be
the $\ell$-cycle $(0,1,\dots,\ell-1)$. Note that the obvious 
action of $w_0$ on the space of functions 
$\{f:\mathbb{Z}/\ell\mathbb{Z}\rightarrow 
\mathbb{Z}/\ell\mathbb{Z}\}$ given by 
$$(w_0f)(x)=f(w_0(x))$$
preserves the set $X$; this can be seen from the following identity: 
$$ [w_0 ({\rm id}+ w)](i) = 
i+w'(i),$$ where $w'$ is the permutation of 
$\mathbb{Z}/\ell\mathbb{Z}$ defined by $w'=1+ww_0$. 
Now, for $w\in S_\ell$, the permutation ${\rm id}+w$
belongs to the fixed point set $X^{w_0}$ if and only if there exists
an unique integer $k\in \mathbb{Z}/\ell\mathbb{Z}$ such that
$$ i+w(i) = k,\ 
0\leq i\leq \ell-1.$$ Therefore, from the
identity (\ref{char_formula_2}), we get
\begin{equation}\label{char_formula_3}
\sum_{w\in S_\ell} 
\chi^{(w)}_\eta(\iota(g)) = \sum_{i=0}^{\ell-1}
\chi_\tau(g^{F^i})^\ell + 
\sum_{{\rm id}+w\in X\setminus X^{w_0}}
\Big(\prod_{i=0}^{\ell-1}\,
\chi_\tau(g^{F^{i+w(i)}})\Big).
\end{equation}
Since $\langle w_0\rangle$ acts freely on $X\setminus X^{w_0}$, there
exists a subset $U\subseteq
X\setminus X^{w_0}$ such that
$$ X\setminus X^{w_0} = 
\coprod_{r=0}^{\ell-1}w_0^rU. $$
Then we have
\begin{equation}\label{char_formula_4}
\sum_{{\rm id}+w\in X\setminus X^{w_0}}\Big(\prod_{i=0}^{\ell-1}\,
\chi_\tau(g^{F^{({\rm id}+w)(i)}})\Big) =
\sum_{r=0}^{\ell-1}\,\sum_{{\rm id}+w\in U}
\,\prod_{i=0}^{\ell-1}\chi_\tau
(g^{F^{[w_0^r({\rm id}+w)](i)}}).
\end{equation}
Recall that $[w_0^r({\rm id}+w)](i) = (i+r) + w(i+r)$,
$0\leq r \leq \ell-1$, and we have
$$ \prod_{i=0}^{\ell-1} \chi_\tau
(g^{F^{[w_0^r({\rm id}+w)](i)}}) = 
\prod_{i=0}^{\ell-1} \chi_\tau
(g^{F^{i+w(i)}}). $$
Using this identity to the equation (\ref{char_formula_4}), we get
$$ \sum_{{\rm id}+w\in X\setminus X^{w_0}}\Big(\prod_{i=0}^{\ell-1}\,
\chi_\tau(g^{F^{i+w(i)}})\Big) =
\ell\sum_{{\rm id}+w\in U}
\,\prod_{i=0}^{\ell-1}\chi_\tau
(g^{F^{i+w(i)}}). $$ Then it
follow from the equation 
\eqref{char_formula_3}, we get
$$ \sum_{w\in S_\ell}
\chi^{(w)}_\eta(\iota(g)) \equiv
\sum_{i=0}^{\ell-1}
\chi_\tau(g^{F^i})^\ell\ \text{mod}\  (\ell).$$
Hence the lemma.
\end{proof}
\subsection{}
Let $G\rtimes \Gamma$ be a finite group such that
$\Gamma\simeq \mathbb{Z}/\ell\mathbb{Z}$, and we fix a generator
$\sigma$ of $\Gamma$.  The following lemma, in certain cases, gives
the trace of $G^\sigma$ action on the Tate cohomology with respect to
the action of $\Gamma$.  This will be useful when we consider the Tate
cohomology of a lattice in Shintani lifts.
\begin{lemma}\label{char_iden_1}
  Let $G$ be a finite group with an automorphism $\sigma$ of order
  $\ell$. Let $\Gamma$ be the group generated by $\sigma$. Let $M$ be
  a $\Lambda[G\rtimes \Gamma]$-module such that $M$ is a free
  $\Lambda$-module of finite rank. Assume that
  $\widehat{H}^1(\Gamma, M) =0$. Then
$$ {\rm Tr}(g|\widehat{H}^0(\Gamma, M))=\overline{{\rm Tr}(g|M)}, $$ 
for all $g\in G^\sigma$.
\end{lemma}
\begin{proof}
  This lemma was observed by Ronchetti (\cite[Theorem
  12]{Ronchetti_BC}).  Since, this lemma is used in a crucial way, we
  give a proof.  We may assume that there is a finite unramified
  extension $K$ of $\mathbb{Q}_\ell$ and a
  $\mathfrak{o}_K[G\rtimes \Gamma]$ module $M_0$ such that $M_0$ is a
  free $\mathfrak{o}_K$ module and
  $M=M_0\otimes_{\mathfrak{o}_K} \Lambda$.  Assume that
  $M_0=\oplus_{i=1}^kN_i$, where $N_i$ is an indecomposable
  $\mathbb{Z}_\ell[\Gamma]$ submodule of $M$. Since
  $\widehat{H}^1(\Gamma, M)=0$, we get that $N_i$ is isomorphic to
  $\mathbb{Z}_\ell$ or $\mathbb{Z}_\ell[\Gamma]$.  Let
  $p_i:N_i\rightarrow M$ and $q_i:M\rightarrow N_i$ be the embedding
  and the projection maps, respectively.  Note that
  $${\rm Tr}(g|M)=\sum_{i=1}^k{\rm Tr}(q_igp_i : N_i\rightarrow
  N_i).$$ The action of $q_igp_i$ on $N_i$ has a unique
  eigenvalue. When $N_i=\mathbb{Z}_\ell$, the element $q_igp_i$ acts
  by an unit $\xi\in \mathbb{Z}_\ell^\ast$ whose mod-$\ell$ reduction
  $\overline{\xi}$ coincides with the action of $\xi$ on the Tate
  cohomology space
  $\widehat{H}^0(\Gamma, \mathbb{Z}_\ell) =
  \mathbb{Z}/\ell\mathbb{Z}$. Let us consider the case where
  $N_i=\mathbb{Z}_\ell[\Gamma]$. Note that
  $\widehat{H}^0(\Gamma, N_i)=0$. If $\lambda$ is the unique
  eigenvalue of $q_igp_i$ on ${N_i}$, then
  ${\rm Tr}(q_igp_i) = \ell.\lambda$. Thus,
$$ \overline{{\rm Tr}(g|_{M})} = 0 = {\rm Tr}(g|_{\widehat{H}^0(\Gamma, M)}). $$
Hence the lemma.
\end{proof}	
\begin{proposition}\label{char_iden_3}
  Let $\Pi$ be an absolutely irreducible generic
  $\mathcal{K}$ representation of ${\rm GL}_n(\mathbb{F}_{q^\ell})$
  such that $\Pi\simeq \Pi^\sigma$, for all
  $\sigma\in {\rm Gal}(\mathbb{F}_{q^\ell}/\mathbb{F}_q)$. Let $\pi$
  be the $\ell$-adic cuspidal representation of
  ${\rm GL}_{n}(\mathbb{F}_q)$ such that
  $\Pi\otimes_{\mathcal{K}} \overline{\mathbb{Q}}_\ell$ is the
  Shintani lift of $\pi$.  Let $\mathcal{L}$ be a
  $\Lambda[{\rm GL}_n(\mathbb{F}_{q^\ell})\rtimes {\rm
    Gal}(\mathbb{F}_{q^\ell}/ \mathbb{F}_q)]$ stable lattice in
  $\Pi$. Then we have
$${\rm Tr}(g|\widehat{H}^0({\rm Gal}(\mathbb{F}_{q^\ell}/\mathbb{F}_q), 
\mathcal{L})) = \overline{\chi_\pi(g)}^\ell,$$
for all $g\in {\rm GL}_n(\mathbb{F}_q)$. (For $x\in \Lambda$, the notation 
$\overline{x}$ denotes the mod-$\ell$ reduction of $x$.)
\end{proposition}
\begin{proof}
Let $\sigma$ be the Frobenius element in ${\rm Gal}(\mathbb{F}_{q^\ell}/
\mathbb{F}_q)$ and let $F$ be the corresponding Frobenius 
endomorphism of ${\rm GL}_n$.
We fix a ${\rm GL}_n(\mathbb{F}_{q^\ell})
\rtimes\langle F\rangle$ stable 
$\Lambda$-lattice $\mathcal{L}$ in $\Pi$. 
Since the Tate cohomology group 
$\widehat{H}^0({\rm Gal}(\mathbb{F}_{q^\ell}/
\mathbb{F}_q), \mathcal{L})$ is 
irreducible and cuspidal 
(Theorem \ref{Tate_bc_cuspidal}), 
it is sufficient to show that
\begin{equation}\label{congru_prop_trace}
    {\rm Tr}(g|{\widehat{H}^0({\rm Gal}(\mathbb{F}_{q^\ell}/
\mathbb{F}_q), \mathcal{L})}) = \overline{\chi_\pi(g)}^\ell,
\end{equation} 
for all $g$ with Jordan decomposition $g=su$, where the
characteristic polynomial of $s$ is of the form 
 $f(X)^r$ for some
irreducible polynomial $f(X)\in\mathbb{F}_q[X]$; let $\mathcal{S}$
be the set of all such elements $g\in {\rm GL}_n(\mathbb{F}_q)$. 
Since the Tate
cohomology group $\widehat{H}^1( {\rm Gal}(\mathbb{F}_{q^\ell}/
\mathbb{F}_q), \mathcal{L})$ is trivial (Theorem
\ref{Tate_bc_cuspidal}), it follows from Lemma \ref{char_iden_1} that
the equality \eqref{congru_prop_trace} is equivalent to the following identity
$$ \overline{\chi_\Pi(g)} = 
\overline{\chi_\pi(g)}^\ell,\, g\in \mathcal{S}.$$ 

We use Deligne--Lusztig character formulas for $\chi_\Pi$ and $\chi_\pi$. 
Let $S_0$ be an $F$-stable split maximal torus of ${\rm GL}_n$. 
Let $T$ be an $F$-stable maximal torus of ${\rm GL}_n$ such that 
$(T, F)$ is isomorphic to $(S_0, {\rm Ad}(w_0)\circ F)$, 
where $w_0$ is an $n$-cycle 
in $W({\rm GL}_n, S_0)$. Let
$\theta:T^F\rightarrow \mathcal{K}^\ast$
be a character in general position such that 
$\pi=(-1)^{n-1}R_{T, \theta}$. Let ${\rm Nr}:T^{F^\ell}\rightarrow T^F$
be the norm map 
$$t\mapsto tt^Ft^{F^2}\dots t^{F^{\ell-1}}.$$
Let $\eta$ be the composite 
$$T^{F^\ell}\xrightarrow{\rm Nr} T^F\xrightarrow{\theta}
\overline{\mathbb{Q}}_\ell^\ast.$$ The Shintani lifting of $\pi$ to
${\rm GL}_{n}(\mathbb{F}_{q^\ell})$ is equal to
$\Pi=(-1)^{n-\ell}R^{(\ell)}_{T,\eta}$ (see
\cite{Kawanaka_Shintani_lift} for a reference). Here
$R^{(\ell)}_{T, \eta}$ is the Deligne--Lusztig induction of the
character $\eta$ of $T^{F^\ell}$ to $G^{F^\ell}$.

First we assume that $n=\ell m$ for some integer $m$.
Note that $T^{F^\ell}$ is identified with
$(\mathbb{F}_{q^n}^\ast)^\ell$ and $T^F$ is identified with
$\mathbb{F}_{q^n}^\ast$ such that ${\rm Nr}$ is given by
$$(t_1, t_2,\dots, t_\ell)\mapsto t_1t_2^F
\dots t_\ell^{F^{\ell-1}},\ t_i\in \mathbb{F}_{q^n}^\ast.$$
Let $P$ be an $F^\ell$-stable parabolic subgroup containing $T$, and
let $M$ be an $F^\ell$-stable Levi subgroup of $P$ containing $T$.  We
can identify the group $M^{F^\ell}$ with
${\rm GL}_{n/\ell}(\mathbb{F}_{q^\ell})^\ell$ and
$\Pi=(-1)^{n-\ell}R^{(\ell)}_{T, \eta}$ with the parabolic induction
$$\tau\times\tau^{F}\times\cdots\times\tau^{F^{\ell-1}},$$
where $\tau$ is the irreducible cuspidal representation of 
${\rm GL}_{n/\ell}(\mathbb{F}_{q^\ell})$, obtained as
the Deligne--Lusztig induction of the character 
$\theta$ of $\mathbb{F}_{q^n}^\ast$ (which is 
considered as $F^{\ell}$-fixed points of 
a $F^\ell$-stable maximal torus of ${\rm GL}_{n/\ell}$
corresponding to an $n/\ell$-cycle).

Let $s\in T^F$ be an element with characteristic polynomial $f(X)^r$,
where $f(X)\in \mathbb{F}_q[X]$ is an irreducible polynomial of degree
$d$. We have $n=dr$. Let $H$ be the centralizer $C_{{\rm GL}_n}(s)$.
Let $u\in H^F$ be an unipotent element.  From the Deligne--Lusztig
character formula (\ref{DL_char_form_1}) and Lemma
\ref{Conjugacy_maximal_torus}, we get that
$$\chi_{\Pi}(su)=Q_{H, T, F^\ell}(u)
\sum_{w\in W(G, T)^{F^\ell}/W(H, T)^{F^\ell}}\eta(wsw^{-1}).$$
We begin with the case where $\ell\mid d$. 
Let $f=f_1f_2\cdots f_\ell$ be the
splitting of $f$ where $f_i(X)\in \mathbb{F}_{q^\ell}[X]$ is an 
irreducible polynomial. With the above identification, the 
element $s\in T^F\subset M^{F^\ell}$ is of the form 
$$s=(s_0,s_1,\dots, s_{\ell-1}), s_i\in \mathbb{F}_{q^{d/\ell}},\ 0\leq i\leq \ell-1,$$
where $s_i$, considered as an element of ${\rm GL}_{n/\ell}(\mathbb{F}_{q^\ell})$,
has characteristic polynomial $f_i(X)^r$. Thus, we have 
$$F(s_0)=s_1, F(s_1)=s_2,\dots, F(s_{\ell-1})=s_0.$$
Note that $H^{F^\ell}$ can be identified with 
${\rm GL}_r(\mathbb{F}_{q^d})^\ell$ and 
the element $u\in H^{F}\subset H^{F^\ell}$ with 
$$(u_1, u_2,\dots, u_\ell)$$
where $u_i$ is the unipotent element of
${\rm GL}_{r}(\mathbb{F}_{q^{d}})$ conjugate to $u$.  Thus we have
$$Q_{H, T, F^\ell}(u)=Q_{H, T, F}(u)^\ell.$$
Let $\gamma$ be the Frobenius element of
${\rm Gal}(\mathbb{F}_{q^d}/\mathbb{F}_{q^\ell})$, then the set of
conjugates
\begin{align*}
    &\{wsw^{-1}:w\in W(G, T)^{F^\ell}/W(H, T)^{F^\ell}\}\\
  =\{(\gamma^{k_0}&(s_{\alpha(0)}),\gamma^{k_1}(s_{\alpha(1)}),\dots,
                    \gamma^{k_{\ell-1}}(s_{\alpha(\ell-1)}) : 0\leq k_i\leq d/\ell,\, \alpha\in S_\ell
\}.
\end{align*}
We observe that 
\begin{align*}
   &\eta((\gamma^{k_0}(s_{\alpha(0)}),\gamma^{k_1}(s_{\alpha(1)}),\dots, 
\gamma^{k_{\ell-1}}(s_{\alpha(\ell-1)})))\\&=
\eta((\gamma^{k_1}(F(s_{\alpha(1)})),\gamma^{k_2}(F(s_{\alpha(2)})),\dots, 
\gamma^{k_{\ell-1}}(F(s_{\alpha(\ell-1)})), \gamma^{k_0}(F(s_{\alpha(0)})))). 
\end{align*}
Let $Y$ be the set 
$$\{(\gamma^{k}(s_{\alpha(0)}),\gamma^{k}(s_{\alpha(1)}),\dots, 
\gamma^k(s_{\alpha(\ell-1)}) : 0\leq k\leq d/\ell,\, \alpha\in S_\ell\}.$$
We deduce that 
$$\sum_{w\in W(G, T)^{F^\ell}/W(H, T)^{F^\ell}}\eta(wsw^{-1})\equiv
\sum_{wsw^{-1}\in \in Y}\eta(wsw^{-1})\ \text{mod}\ (\ell).$$
Using Lemma \ref{char_iden_2}, we get that 
$$Q_{H, T, F}(u)^\ell\sum_{wsw^{-1}\in \in Y}\eta(wsw^{-1})\equiv
Q_{H, T, F}(u)^\ell\sum_{0\leq i\leq d}\theta(s^{q^i})^{\ell}\ \text{mod}\ (\ell).$$
This shows the proposition in this case. 

We now assume that $(d, \ell)=1$. In this case, we have 
$$\chi_{\pi}(su)=Q_{H, T,F^\ell}(u)\sum_{w\in W(G, T)^{F^\ell}/W(H, T)^{F^\ell}}\eta(wsw^{-1}).$$
Note that 
$$Y=\{wsw^{-1}:w\in W(G, T)^{F^\ell}/W(H, T)^{F^\ell}\}=
\{(s^{F^{\alpha(0)}}, s^{F^{\alpha(1)}},\dots, 
s^{F^{\alpha(\ell-1)}}) : \alpha\in S_\ell\}.
$$
Using Lemma \ref{char_iden_2}, we get that 
$$\sum_{y\in Y}\eta(y)=\sum_{0\leq i\leq d}\theta(s^{q^i})^\ell\ \text{mod}\ (\ell).$$
We must now analyse the Green function, for which we first recall the following 
identifications. 
$$H^F={\rm GL}_r(\mathbb{F}_{q^d}), \ H^{F^\ell}={\rm
  GL}_{r}(\mathbb{F}_{q^{d\ell}}), T^F =\mathbb{F}_{q^n}^\ast,
T^{F^\ell}=(\mathbb{F}_{q^n}^\ast)^\ell.$$ Assume that
$\mu=(\mu_1, \mu_2,\dots, \mu_k)$ is the partition associated to the
conjugacy class of $u$ in $H^F$. Then $\mu$ is the partition of the
$H^{F^\ell}$ conjugacy class of $u$. The partition associated with
$T^F$ is $(n)$ and $(n/\ell, n/\ell,\dots, n/\ell)$ is the partition
of $T^{F^\ell}$.  Let $w\in S_n$ be an $n$-cycle. It follows from
Springer that
$$Q_{H, T, F}(u)=\sum_{i\geq 0}a_i(w, u)q^i,$$
and 
$$Q_{H, T, F^\ell}(u)=\sum_{i\geq 0}a_i(w^\ell, u)q^{i\ell},$$
where $g\mapsto a_i(g, u)$, for $g\in S_n$, is a character of a representation of 
$S_n$. Clearly, we have 
$$a_i(w, u)^\ell\equiv a_i(w^\ell, u)\ \text{mod}\ (\ell).$$
From which we get that
$$Q_{H, T, F^\ell}(u)\equiv Q_{H, T, F}(u)^\ell\ \text{mod}\ (\ell).$$
We thus conclude in all cases 
$$\chi_{\Pi}(su)\equiv \chi_\pi(su)^\ell\ \text{mod}\ (\ell).$$

In the case where $(\ell, n)=1$, the above proposition is proved by
Ronchetti (\cite[Theorem 13]{Ronchetti_BC}). For the uniformity of the
proof, we treat this case using our method. For $g\in \mathcal{S}$, we
have
$$R_{T, \eta}^{(\ell)}(g)=Q_{H, T, F^\ell}(u)
\sum_{w\in W(G, T)^{F^\ell}/W(H, T)^{F^\ell}}\theta(wsw^{-1})$$
and in this case, we have 
$$H^F={\rm GL}_r(\mathbb{F}_q),\ 
H^{F^\ell}={\rm GL}_r(\mathbb{F}_{q^{d\ell}}),\ 
T^F=\mathbb{F}_{q^\ell},\ T^{F^\ell}=\mathbb{F}_{q^{n\ell}}$$
and we have 
$$Q_{H, T, F^\ell}(u)=\sum_{i\geq 0}a_i(w, u)q^{i\ell}\equiv 
(\sum_{i\geq 0}a_i(w, u)q^i)^\ell\ \text{mod}\ (\ell).$$
Then we have 
$$\sum_{w\in W(G, T)^{F^\ell}/W(H, T)^{F^\ell}}
\theta(wsw^{-1})=
\sum_{i=1}^{d}\theta(s^{q^{i\ell}})\equiv 
[\sum_{i=1}^d\theta(s^{q^i})]^\ell\ \text{mod}\ (\ell)$$ 
\end{proof}
We now deduce our main theorem.
\begin{theorem}\label{TV_finite_field}
  Let $\Pi$ be an absolutely irreducible generic
  $\mathcal{K}$-representation of ${\rm
    GL}_n(\mathbb{F}_{q^\ell})$. Let $\pi$ be an irreducible
  $\ell$-adic cuspidal representation of ${\rm GL}_n(\mathbb{F}_q)$
  such that $\Pi\otimes_{\mathcal{K}} \overline{\mathbb{Q}}_\ell$ is
  the Shintani lift of $\pi$.  Then for any
  ${\rm GL}_n(\mathbb{F}_{q^\ell})\rtimes {\rm
    Gal}(\mathbb{F}_{q^\ell}/\mathbb{F}_q)$ stable $\Lambda$-lattice
  $\mathcal{L}$ in $\Pi$, the Tate cohomology
  $\widehat{H}^0({\rm
    Gal}(\mathbb{F}_{q^\ell}/\mathbb{F}_q),\mathcal{L})$ is isomorphic
  to the Frobenius twist $r_\ell(\pi)^{(\ell)}$.
\end{theorem}
\begin{proof}
  To prove the theorem, we use the fact that two irreducible
  $\ell$-modular representations of a finite group are isomorphic if
  and only if they have the same character (\cite[Corollary
  7.21]{Lam_noncommutative_rings}). Using Proposition
  \ref{char_iden_3}, we get the character identity
$$ {\rm Tr}(g|_{\widehat{H}^0({\rm
    Gal}(\mathbb{F}_{q^\ell}/\mathbb{F}_q),\mathcal{L})})
= \overline{\chi_\pi(g)}^\ell, $$
for all $g\in {\rm GL}_n(\mathbb{F}_q)$. Note that the Tate cohomology
group
$\widehat{H}^0({\rm
  Gal}(\mathbb{F}_{q^\ell}/\mathbb{F}_q),\mathcal{L})$ and the
Frobenius twist $r_\ell(\pi)^{(\ell)}$ are both irreducible. Thus, the
Tate cohomology
$\widehat{H}^0({\rm
  Gal}(\mathbb{F}_{q^\ell}/\mathbb{F}_q),\mathcal{L})$ is isomorphic
to $r_\ell(\pi)^{(\ell)}$ as ${\rm GL}_n(\mathbb{F}_q)$
representations. Hence the theorem.
\end{proof}

\section{Lifting to 
\texorpdfstring{$p$}{}-adic groups}
In this section, we give an application of Theorem
\ref{TV_finite_field} to the depth zero case of irreducible smooth
representations of ${\rm GL}_n(F)$, where $F$ is a non-Archimedean
local field.
\subsection{Local base change lifting}
For a non-Archimedean local field $K$, let $\mathcal{A}(n,K)$ be the
set of isomorphism classes of irreducible $\ell$-adic smooth
representations of ${\rm GL}_n(K)$.  Let $\mathcal{G}(n,K)$ be the set
of isomorphism classes of semisimple $n$-dimensional $\ell$-adic
Weil--Deligne representations of the Weil group $\mathcal{W}_K$. The
$\ell$-adic local Langlands correspondence is a bijection
$$ {\rm LLC}_K:\mathcal{A}(n,K) 
\rightarrow \mathcal{G}(n,K), $$ which is uniquely characterized by
certain identities of local constants that one attaches to the
elements of $\mathcal{A}(n, K)$ and $\mathcal{G}(n,K)$ (see
\cite{harris_taylor}, \cite{henniart_une_preuve},
\cite{scholze_llc}). Moreover, the set of irreducible $\ell$-adic
cuspidal representations of ${\rm GL}_n(K)$ are mapped onto the set of
$n$-dimensional irreducible $\ell$-adic representations of
$\mathcal{W}_K$ via ${\rm LLC}_K$. Note that the classical local
Langlands correspondence is a bijection between the set of isomorphism
classes of irreducible smooth complex representations of
${\rm GL}_n(K)$ and the set of isomorphism classes of $n$-dimensional complex
semisimple Weil--Deligne representations. To get a correspondence over
$\overline{\mathbb{Q}}_\ell$, one twists the original correspondence
by the character $\nu^{(1-n)/2}$, where $\nu$ is the composition of
the determinant character of ${\rm GL}_n(K)$ and the absolute value of
$K$. For details, see \cite[Conjecture 4.4, Section
4.2]{clozel_motifs} and \cite[Section 7]{henniart_bordeaux}.

Say $F$ be a non-Archimedean local field, and let $E/F$ be a finite
Galois extension such that $[E:F]=\ell$.  Let $\pi$ be an irreducible
$\ell$-adic representation of ${\rm GL}_n(F)$. Let $\Pi$ be an
irreducible $\ell$-adic representation of ${\rm GL}_n(E)$ such that
\begin{center}
  ${\rm res}_{\mathcal{W}_E} \big({\rm LLC}_F(\pi)\big) \simeq {\rm
    LLC}_E(\Pi)$.
\end{center}
The representation $\Pi$ is called the base change lifting of
$\pi$, and we write $\Pi = {\rm bc}_{E/F}(\pi)$.  In this case, $\Pi$ is
isomorphic to $\Pi^\sigma$ for all $\sigma \in{\rm Gal}(E/F)$.
\subsection{}\label{comp_LLC_Shintani}
In this part, we discuss the compatibility of the local base change
and Shintani lift.
Let $K_0(F)$ be the compact open subgroup ${\rm GL}_n(\mathfrak{o}_F)$, and
let $K_1(F)$ be the pro-$p$ subgroup
$$ {\rm Id}_n + \varpi_FM_n(\mathfrak{o}_F), $$
where $M_n(\mathfrak{o}_F)$ is the ring of $n\times n$ matrices with
entries in the ring of integers $\mathfrak{o}_F$, and $\varpi_F$ is a
uniformiser of $F$. The quotient group $K_0(F)/K_1(F)$ is isomorphic
to ${\rm GL}_n(\mathbb{F}_q)$, where $\mathbb{F}_q$ is the residue
field of $F$.  An irreducible smooth representation $(\pi,V)$ of
${\rm GL}_n(F)$, where $V$ is a vector space over
$\overline{\mathbb{Q}}_\ell$ or $\mathcal{K}$ or
$\overline{\mathbb{F}}_\ell$, is said be of depth-zero if there exists
a non-zero vector in $V$ that is fixed under the action of $K_1(F)$.
 Let $\mathcal{A}(n,F)_0$ be the set of isomorphism
classes of irreducible smooth depth-zero $\ell$-adic representations
of ${\rm GL}_n(F)$, and let $\mathcal{A}(n,\mathbb{F}_q)$ be the set
of isomorphism classes of irreducible $\ell$-adic representations of
${\rm GL}_n(\mathbb{F}_q)$. In the article \cite[Section A.1.3]{Zink_Schneider},
Schneider--Zink defined a map
$$ \varphi_0^F : \mathcal{A}(n,F)_0
\longrightarrow \mathcal{A}(n,\mathbb{F}_q). $$ 
Let $E$ be a finite unramified extension of $F$ with $[E:F]=\ell$.
In \cite[Appendix A, equation (A.7)]{Zink_Silberger_Shintani_LLC},
Zink--Silberger proved the following identity
\begin{equation}\label{Comp_llc_shintani}
\varphi_0^E \circ {\rm bc}_{E/F}
= {\rm bc}_{\mathbb{F}_{q^\ell}/\mathbb{F}_q} 
\circ \varphi_0^F,
\end{equation}
which shows the compatibility of local base change with Shintani lifting.

Let $\Gamma$ be the Galois group ${\rm Gal}(E/F)$ and $\sigma$ be a
generator of ${\rm Gal}(E/F)$. Let $(\Pi,V)$ be an absolutely
irreducible depth-zero integral generic $\mathcal{K}$ representation
of ${\rm GL}_n(E)$ such that $\Pi\simeq\Pi^\sigma$. Let $(\pi,W)$ be
the irreducible depth-zero integral $\ell$-adic cuspidal
representation of ${\rm GL}_n(F)$ such that
$\Pi\otimes_{\mathcal{K}}\overline {\mathbb{Q}}_\ell$ is the base
change lift of $\pi$. From the definition, we have
$\varphi_0^E (\Pi) = \Pi^{K_1(E)}$ and
$\varphi_0^F(\pi) = \pi^{K_1(F)}$. Using the identity
(\ref{Comp_llc_shintani}), we get that the irreducible
${\rm GL}_n(\mathbb{F}_{q^\ell})$ representation
$\Pi^{K_1(E)}\otimes_{\mathcal{K}} \overline{\mathbb{Q}}_\ell$ is the
Shintani lift of the irreducible $\ell$-adic cuspidal representation
$\pi^{K_1(F)}$ of ${\rm GL}_n(\mathbb{F}_{q})$. Using these
observations, we now prove the following theorem.
\begin{theorem}
  Let $F$ be a non-Archimedean local field with residue characteristic
  $p$, and let $E/F$ be a finite unramified Galois extension with
  $[E:F]=\ell$, where $\ell$ and $p$ are distinct odd primes. Let
  $\Pi$ be a depth-zero, integral, generic,
  $\mathcal{K}$-representation of ${\rm GL}_n(F)$ such that $\Pi$ is
  absolutely irreducible and $\Pi\simeq \Pi^\sigma$. Let $\pi$ be the
  depth-zero, irreducible, integral, $\ell$-adic cuspidal
  representation of ${\rm GL}_n(F)$ such that
  $\Pi\otimes_{\mathcal{K}} \overline{\mathbb{Q}}_\ell$ is the base
  change lifting of $\pi$. Then
\begin{enumerate}
\item For any ${\rm GL}_n(E)\rtimes {\rm Gal}(E/F)$ stable
  $\Lambda$-lattice $\mathcal{L}$ of $\Pi$, the Tate cohomology group
  $\widehat{H}^0({\rm Gal}(E/F), \mathcal{L})$ is isomorphic to the
  Frobenius twist $r_\ell(\pi)^{(\ell)}$.
\item If the mod-$\ell$ reduction $r_\ell(\Pi)$ is irreducible, then
$\widehat{H}^0({\rm Gal}(E/F), r_\ell(\Pi))$ is isomorphic to $r_\ell(\pi)^{(\ell)}$.
\end{enumerate}
\end{theorem}
\begin{proof}
  Let $\mathcal{V}$ be either the
  $\Lambda[{\rm GL}_n(E)\rtimes \Gamma]$ module $\mathcal{L}$
  or the $\overline{\mathbb{F}}_\ell[{\rm GL}_n(E) \rtimes\Gamma]$
  module $r_\ell(\Pi)$. Consider the compact open subgroup $K_1(E)$ of
  ${\rm GL}_n(E)$. We have the following decomposition
$$ \mathcal{V} = \mathcal{V}^{K_1(E)} \bigoplus \mathcal{V}(K_1(E)),$$
where $\mathcal{V}(K_1(E))=\oplus \mathcal{V}_\rho$, where $\rho$ is a
non-trivial representation of $K_1(E)$ and $\mathcal{V}_\rho$ is the
$\rho$ isotypic component in $\mathcal{V}$. The space
$\mathcal{V}(K_1(E))$ is stable under the action of ${\rm Gal}(E/F)$.
Using the compatibility of Tate cohomology with finite direct sums, we
get
\begin{equation}\label{direct_sum}
  \widehat{H}^0(\mathcal{V}) =
  \widehat{H}^0(\mathcal{V}^{K_1(E)}) \bigoplus
  \widehat{H}^0\big(\mathcal{V}(K_1(E))\big).
\end{equation}
Note that the ${\rm GL}_n(\mathbb{F}_{q^\ell})$ representation
$\Pi^{K_1(E)}\otimes_{\mathcal{K}} \overline{\mathbb{Q}}_\ell$ is the
Shintani lift of the irreducible cuspidal representation
$\pi^{K_1(F)}$ of ${\rm GL}_n(\mathbb{F}_q)$. Then it follows from Theorem
\ref{TV_finite_field} that $\widehat{H}^0(\mathcal{V}^{K_1(E)})$ is
${\rm GL}_n(\mathbb{F}_q)$-isomorphic to
$r_\ell(\pi^{K_1(F)})^{(\ell)}$. Since $\widehat{H}^0(\mathcal{V})$ is
irreducible and cuspidal (Theorem \ref{Tate_bc_cuspidal} and Remark
\ref{Remark_Tate_bc_cuspidal}), and the space
$\widehat{H}^0(\mathcal{V})^{K_1(F)}$ contains
$(r_\ell(\pi)^{(\ell)})^{K_1(F)}$, we get that
$\widehat{H}^0(\mathcal{V})$ is isomorphic to $r_\ell(\pi)^{(\ell)}$.
This completes the proof.
\end{proof}
\subsection{Tate cohomology of 
depth-zero generic representations} 
This part aims to prove Treumann--Venkatesh's conjecture for
$\ell$-adic integral, depth-zero, generic representations of
${\rm GL}_n(F)$ without the assumption on the prime $\ell$ made in
\cite{nadimpalli2024tate}, i.e., $\ell$ does not divide
$\lvert{\rm GL}_{n-1}(\mathbb{F}_q) \rvert$ for $n\geq 3$. We first
introduce some notations.
\subsubsection{Supercuspidal support}
Let $\lambda=(n_1,\dots,n_r)$ be a partition of $n$. Let $P_\lambda$
be the parabolic subgroup of ${\rm GL}_n$ consisting of block upper
triangular matrices of type $\lambda$. Let $M_\lambda$ be the Levi
subgroup of $P_\lambda$ which is isomorphic to
${\rm GL}_{n_1}\times\cdots \times {\rm GL}_{n_r}$. For each
$1\leq i\leq r$, let $\xi_i$ be an irreducible smooth supercuspidal
$R$-representation of ${\rm GL}_{n_i}(F)$, where $R$ denotes either
$\overline{\mathbb{Q}}_\ell$ or $\mathcal{K}$ or
$\overline{\mathbb{F}}_\ell$. Let $\xi$ be the $M_\lambda(F)$
representation $\xi_1\otimes\cdots\otimes\xi_r$. Let $\pi$ be a smooth
irreducible $R$-representation of ${\rm GL}_n(F)$. There exists a pair
$(M_\lambda(F),\xi)$ as defined above such that $\pi$ is a
sub-quotient of the normalised parabolic induction
$i_{P_\lambda(F)}^{{\rm GL}_n(F)}\xi$.  Moreover, the pair
$(M_\lambda(F),\xi)$ is unique up to ${\rm GL}_n(F)$-conjugacy and is
called the supercuspidal support of $\pi$.  When
$R=\overline{\mathbb{Q}}_\ell$, this follows from the Frobenius
reciprocity (see \cite{BZ_1} for instance).  When
$R=\overline{\mathbb{F}}_\ell$, this is a deep result, proved by
Vign\'eras (\cite{Vigneras_induced}).
\subsubsection{}
Let $\pi$ be an $\ell$-adic integral generic
representation of ${\rm GL}_n(F)$,
and let $(M(F),\xi)$ be the 
supercuspidal support of $\pi$. We assume that $\pi$ is a 
subrepresentation of
$i_{P(F)}^{{\rm GL}_n(F)}\xi$, where $P$ is 
a parabolic subgroup of ${\rm GL}_n$ whose 
Levi component is $M$. Let $\Pi$ be 
an absolutely irreducible,
integral, generic $\mathcal{K}$-representation 
of ${\rm GL}_n(E)$ such that
$\Pi\otimes_{\mathcal{K}}
\overline{\mathbb{Q}}_\ell$ is the base change
lift of $\pi$. Then $\Pi$ is a 
subrepresentation of
$i_{P(E)}^{{\rm GL}_n(E)}\eta$ with
$\eta = \eta_1\otimes\cdots\otimes\eta_r$, 
where $\eta_i$ is an
absolutely irreducible integral generic $\mathcal{K}$-representation
of ${\rm GL}_{n_i}(E)$ 
and
$\eta_i\otimes_{\mathcal{K}}
\overline{\mathbb{Q}}_\ell$ is the base
change lift of $\xi_i$ for each $1\leq i\leq r$.
\subsubsection{}
For each $1\leq i\leq r$, let $\mathcal{L}_i$ be a
$\Lambda[{\rm GL}_{n_i}(E)\rtimes \Gamma]$ stable lattice in
$\eta_i$. Let $\mathcal{L}$ be the module
$\mathcal{L}_1\otimes\cdots\otimes \mathcal{L}_r$, which is a
$M(E)\rtimes\Gamma$ stable lattice in $\eta$. By
\cite[Chapter I, Section 9.3]{Vigneras_modl_book},
$i_{P(E)}^{{\rm GL}_n(E)}\mathcal{L}$ is a ${\rm GL}_n(E)$ stable
$\Lambda$-lattice in $i_{P(E)}^{{\rm GL}_n(E)}\eta$--which is also
stable under the action of $\Gamma$ on $i_{P(E)}^{{\rm GL}_n(E)}\eta$,
defined by
$$ (\sigma f)(g) = f(\sigma^{-1}(g)) $$
for all $\sigma\in \Gamma$, $g\in {\rm GL}_n(E)$, and
$f\in i_{P(E)}^{{\rm GL}_n(E)}\eta$.  Let $\mathcal{L}'$ be the
$\Lambda[{\rm GL}_n(E)\rtimes\Gamma]$ stable lattice
$\Pi\cap i_{P(E)}^{{\rm GL}_n(E)} \mathcal{L}$ in $\Pi$.
\begin{theorem}
Let $\pi$ be a smooth, depth-zero,
integral, $\ell$-adic generic 
representation of ${\rm GL}_n(F)$, with
$J_\ell(\pi)$ the unique generic 
sub-quotient of the mod-$\ell$ reduction
of $\pi$. Let $\Pi$ be an 
absolutely irreducible, depth-zero,
integral, generic $\mathcal{K}$-representation 
of ${\rm GL}_n(E)$ such that
$\Pi\otimes_{\mathcal{K}} 
\overline{\mathbb{Q}}_\ell$ is the base
change lift of $\pi$. Then 
the Frobenius twist $J_\ell(\pi)^{(\ell)}$ is
the unique generic sub-quotient 
of the Tate cohomology group
$\widehat{H}^0({\rm Gal}(E/F), \mathcal{L}')$.
\end{theorem}
\begin{proof}
  Let $(M(F),\xi)$ be the pair such that $\pi$ is a subrepresentation
  of $i_{P(F)}^{{\rm GL}_n(F)}\xi$, where
  $$ M = {\rm GL}_{n_1}\times\cdots\times {\rm GL}_{n_r},\,\,\, \xi =
  \xi_1\otimes\cdots\otimes \xi_r, $$ and each $\xi_i$ is an
  irreducible, integral, $\ell$-adic cuspidal representation of
  ${\rm GL}_{n_i}(F)$. Now,
  $\Pi\otimes_{\mathcal{K}}\overline {\mathbb{Q}}_\ell$, being the
  base change lift of $\pi$, is a subrepresentation of
  $i_{P(E)}^{G_n(E)}\eta$ with
  $\eta = \eta_1\otimes\cdots\otimes \eta_r$, where each $\eta_i$ is
  an absolutely irreducible, integral, generic
  $\mathcal{K}$-representation of ${\rm GL}_{n_i}(E)$ such that
  $\eta_i\otimes_{\mathcal{K}} \overline{\mathbb{Q}}_\ell$ is the base
  change lift of $\xi_i$.  Then we have the following commutative
  diagram
$$
\xymatrixrowsep{0.2in}
\xymatrixcolsep{0.15in}
\xymatrix{
\widehat{H}^0(\mathcal{L}') \ar[dd]_{}  
\ar[rr]^{\varphi}  && \widehat{H}^0
(i_{P(E)}^{{\rm GL}_n(E)}\mathcal{L}) 
\ar[dd]^{} \\\\
\widehat{H}^0(\mathcal{L}')_
{N_n(F),\overline{\Theta}_F^\ell} 
\ar[dd]_{f_1} \ar[rr]^{}  && 
\widehat{H}^0
(i_{P(E)}^{{\rm GL}_n(E)}\mathcal{L})
_{N_n(F),\overline{\Theta}_F^\ell} 
\ar[dd]^{f_2} \\\\
\widehat{H}^0(\mathcal{L}'_{N_n(E),\Theta_E}) 
\ar[rr]_{} && \widehat{H}^0\big
((i_{P(E)}^{{\rm GL}_n(E)}\mathcal{L})
_{N_n(E),\Theta_E}\big)
}
$$
where $\varphi$ is a ${\rm GL}_n(F)$ equivariant map induced from the
injection
$\mathcal{L'}\hookrightarrow i_{P(E)}^{{\rm GL}_n(E)}\mathcal{L}$.
Note that the twisted Jacquet modules
$\mathcal{L}'_{N_n(E), \Theta_E}$ and
$(i_{P(E)}^{{\rm GL}_n(E)}\mathcal{L}) _{N_n(E),\Theta_E}$ are free
$\Lambda$-modules of rank $1$. Also note that the vertical maps $f_1$
and $f_2$ are non-zero and hence isomorphisms. Therefore, the map
$\varphi$ takes the generic subquotients of
$\widehat{H}^0(\mathcal{L}')$ to the generic subquotients of
$\widehat{H}^0(i_{P(E)}^{{\rm GL}_n(E)}\mathcal{L})$.  Recall that
$\eta_i\otimes_{\mathcal{K}} \overline{\mathbb{Q}}_\ell$ is the base
change lifting of $\xi_i$, and $\mathcal{L}_i$ is a
${\rm GL}_{n_i}(E)\rtimes {\rm Gal}(E/F)$ stable $\Lambda$-lattice in
$\eta_i$, for each $1\leq i\leq r$. Then using Theorem
\ref{Tate_bc_cuspidal}, we get that the Tate cohomology group
$\widehat{H}^0(\mathcal{L}_i)$ is isomorphic to
$r_\ell(\xi_i)^{(\ell)}$, and the equation
(\ref{Tate_restriction_map}) gives the isomorphism
$$
\widehat{H}^0(i_{P(E)}^{{\rm GL}_n(E)}
\mathcal{L}) \simeq 
i_{P(F)}^{{\rm GL}_n(F)}r_\ell(\xi)^{(\ell)}.
$$
Thus, the supercuspidal support
of the Frobenius twist $J_\ell(\pi)^{(\ell)}$
is same as the supercuspidal support of the
generic subquotients of 
the Tate cohomology
$\widehat{H}^0(\mathcal{L}')$.
Hence, $J_\ell(\pi)^{(\ell)}$
is a subquotient of
$\widehat{H}^0(\mathcal{L}')$. 
This completes the proof.
\end{proof}

\section{The unipotent cuspidal representation of
\texorpdfstring{${\rm Sp}_4(\mathbb{F}_q)$}{}}\label{sec_theta10}
Finite classical groups can have at most one irreducible
unipotent cuspidal representation. 
Let $G$ be a connected reductive group defined over $\mathbb{F}_q$,
and let $F:G\rightarrow G$ be the Frobenius endomorphism.  
Recall that an irreducible
representation of a finite reductive group $G^F$
is said to be unipotent if it occurs in a 
Deligne--Lusztig virtual character $\pm R_{T, 1}$, 
where $T$ is an $F$-stable torus in $G$. 
When $G$ is a classical group, if an irreducible  unipotent cuspidal 
representation exists, then it is invariant under automorphisms of the 
underlying algebraic group. Additionally, we have
\begin{enumerate}
    \item  Unipotent cuspidal representations of finite classical groups 
    are defined over $\mathbb{Q}$ (\cite{Lusztig_rationality_unipotent}).
    \item The mod-$\ell$ reduction of an unipotent cuspidal representation is 
irreducible (\cite[Theorem 1]{Malle_Dudas_unipotent_cuspidal}).
\end{enumerate}
So, the case of unipotent cuspidal representations for finite
classical groups is a natural testing ground for Treumann--Venkatesh
conjectures. In this section, we treat the group
${\rm Sp}_4(\mathbb{F}_q)$ which has a unique cuspidal unipotent
representation for all $q$.  Recall that ${\rm Sp}_4(\mathbb{F}_q)$ is
defined as the group of all $g\in {\rm GL}_4(\mathbb{F}_q)$ such that
$gJg^{t} = J$ (Note that $g^t$ is the transpose of $g$), where
$$ J =\begin{pmatrix}
	0 & I_2\\
	-I_2 & 0
\end{pmatrix}.$$
Let $F$ be the Frobenius endomorphism of ${\rm Sp}_4/\mathbb{F}_q$. 
In the famous article \cite{Srinivasan-Sp4}, 
B. Srinivasan discovered an irreducible 
cuspidal representation of the group ${\rm Sp}_4
(\mathbb{F}_q)$, which does not have Whittaker model, and this 
representation is denoted by $\theta_{10}$. The representation 
$\theta_{10}$ is an unipotent cuspidal representation. In what 
follows we assume that $\ell\neq 2$. 
\subsection{Mirabolic restriction of $\theta_{10}$}
In \cite{Golfand_Sp4}, Gol'fand proved 
that $\theta_{10}$, like in the case 
of irreducible cuspidal representations 
of ${\rm GL}_n(\mathbb{F}_q)$, has a mirabolic
restriction--which is irreducible. 
To recall these results, we begin by defining some subgroups. 
 Let $P$ be a Siegel parabolic subgroup of 
${\rm Sp}_4$, whose Levi subgroup $M$ 
and the unipotent radical $U$ such that
$$ 
M(\mathbb{F}_q) = \bigg\{
\begin{pmatrix}
	g & 0\\
	0 & w\, ^tg^{-1}w
\end{pmatrix}: g\in {\rm GL}_2(\mathbb{F}_q)\bigg\},\,\,\,\,
U(\mathbb{F}_q) = \bigg\{
\begin{pmatrix}
	I_2 & X\\
	0 & I_2
\end{pmatrix}: X\in M_2(\mathbb{F}_q),\, ^tX=X\bigg\},
$$
where $w=
\begin{pmatrix}
	0 & 1\\
	1 & 0
\end{pmatrix}$ and $I_2$ is the $2\times 2$ 
identity matrix. Fix a non-trivial additive 
character $\psi:\mathbb{F}_q\rightarrow 
\overline{\mathbb{Q}}_\ell^\ast$,
and an element $\tau\in\mathbb{F}_q^\times\setminus (\mathbb{F}_q^\times)^2$. 
We set
$$z=\begin{pmatrix}
	0 & 1\\
	-\tau & 0
      \end{pmatrix}.$$ Consider the character
      $\psi_z:U(\mathbb{F}_q) \rightarrow
      \overline{\mathbb{Q}}_\ell^\ast$, sending
$$\begin{pmatrix}
	I_2 & X\\
	0 & I_2
\end{pmatrix}\mapsto \psi({\rm Tr}(zX)).$$
Let $S(\mathbb{F}_q)$ be the stabilizer of $\psi_z$ in
$M(\mathbb{F}_q)$ and let
$$Y(\mathbb{F}_q) = S(\mathbb{F}_q)\,U(\mathbb{F}_q).$$ 
Then $\psi_z$ induces a character
$\widetilde{\psi_z} : Y(\mathbb{F}_q)\rightarrow
\overline{\mathbb{Q}}_\ell^\ast$, defined as
$$ \widetilde{\psi_z}(su) = 
{\rm det}(s)\,\psi_z(u), $$ for $s\in S(\mathbb{F}_q)$ and
$u\in U(\mathbb{F}_q)$. The induced representation
${\rm Ind}_{Y(\mathbb{F}_q)}^{P(\mathbb{F}_q)} (\widetilde{\psi_z})$
is irreducible and from \cite[Proposition 1]{Golfand_Sp4} we get
that the restriction ${\rm res}_{P(\mathbb{F}_q)}(\theta_{10})$ is
isomorphic to
${\rm Ind}_{Y(\mathbb{F}_q)} ^{P(\mathbb{F}_q)}
(\widetilde{\psi_z})$. Moreover, the mod-$\ell$ reduction of
${\rm Ind}_{Y(\mathbb{F}_q)}^ {P(\mathbb{F}_q)} (\widetilde{\psi_z})$
is irreducible and hence the mod-$\ell$ reduction
$r_\ell(\theta_{10})$ is also irreducible.
\subsection{}
Let $\sigma$ be the Frobenius element in
${\Gamma = \rm Gal}(\mathbb{F}_{q^\ell} /\mathbb{F}_q)$. We denote by
$\Pi$ and $\pi$ the unipotent $\ell$-adic cuspidal representation of
${\rm Sp}_4(\mathbb{F}_{q^\ell})$ and ${\rm Sp}_4(\mathbb{F}_q)$,
respectively. By uniqueness of $\Pi$ we have $\Pi\simeq
\Pi^\sigma$. Let $\eta$ be the additive character of
$\mathbb{F}_{q^\ell}$ given by the composition
$\psi\circ {\rm Tr}_{\mathbb{F}_{q^\ell}/\mathbb{F}_q}$, where
${\rm Tr}_{\mathbb{F}_{q^\ell}/\mathbb{F}_q}$ is the trace
function. Since $\ell$ is odd, we have
$\tau\in \mathbb{F}_{q^\ell} ^\times\setminus
(\mathbb{F}_{q^\ell}^\times)^2$. Let $\eta_z$ and $\widetilde{\eta}_z$
be the corresponding characters of $U(\mathbb{F}_{q^\ell})$ and
$Y(\mathbb{F}_{q^\ell})$, respectively.  Then, following the above
constructions, we get an irreducible $P(\mathbb{F}_{q^\ell})$
representation
${\rm Ind}_{Y(\mathbb{F}_{q^\ell})} ^{P(\mathbb{F}_{q^\ell})}
(\widetilde{\eta_z})$ such that the restriction
${\rm res}_ {P(\mathbb{F}_{q^\ell})}(\Pi)$ is isomorphic to
${\rm Ind}_{Y(\mathbb{F}_{q^l})}^{P(\mathbb{F}_{q^\ell})}
(\widetilde{\eta_z})$. We use the notations $\rho$ and $\rho_\ell$ for
the representations
${\rm Ind}_{Y(\mathbb{F}_q)} ^{P(\mathbb{F}_q)} (\widetilde{\psi_z})$
and
${\rm Ind}_{Y(\mathbb{F}_{q^\ell})}^{P(\mathbb{F}_{q^\ell})}
(\widetilde{\eta_z})$, respectively.
\subsubsection{}
Let $X$ and $X_\ell$ be the coset spaces
$Y(\mathbb{F}_{q})\backslash P(\mathbb{F}_{q})$ and
$Y(\mathbb{F}_{q^\ell})\backslash P(\mathbb{F}_{q^\ell})$,
respectively. Since $Y(\mathbb{F}_{q^\ell})$ is $\Gamma$ stable
subgroup of $P(\mathbb{F}_{q^\ell})$, we have the long exact sequence
of non-abelian cohomology (\cite[Appendix, Proposition
1]{Galois_cohomology_Serre})
\begin{equation}\label{exact_seq_1}
0\rightarrow Y(\mathbb{F}_{q})\rightarrow P(\mathbb{F}_{q})
\rightarrow X_\ell^\sigma\rightarrow
H^1(\Gamma,Y(\mathbb{F}_{q^\ell}))
\rightarrow H^1(\Gamma,P(\mathbb{F}_{q^\ell}))
\end{equation}
Note that $S(\mathbb{F}_{q^\ell})$ is a $\Gamma$ stable subgroup of
$Y(\mathbb{F}_{q^\ell})$ with
$S(\mathbb{F}_{q^\ell})^\sigma = S(\mathbb{F}_q)$, and we get the
following the long exact sequence of non-abelian cohomology
$$ 0\rightarrow S(\mathbb{F}_{q})\rightarrow
Y(\mathbb{F}_{q})\rightarrow U(\mathbb{F}_q)
\rightarrow H^1(\Gamma,S(\mathbb{F}_{q^\ell})) \rightarrow
H^1(\Gamma,Y(\mathbb{F}_{q^\ell})) \rightarrow
H^1(\Gamma,U(\mathbb{F}_{q^\ell})).$$ From the above exact sequence,
we observe that $H^1(\Gamma, S(\mathbb{F}_{q^\ell}))$ is isomorphic to
$H^1(\Gamma, Y(\mathbb{F}_{q^\ell}))$.  The group $S$ is precisely
given by
$$ \bigg\{\begin{pmatrix}
g & 0\\
0 & w\,^tg^{-1}w
\end{pmatrix}\in M(\overline{\mathbb{F}}_q) : gzg^\ast = z\bigg\}, $$
where $g^\ast$ is the transposition of $g$ with respect to the
second diagonal. Note that the determinant map $S\rightarrow \{\pm 1\}$, given by 
$g\mapsto {\rm det}(g)$ is surjective. 
The kernel of this determinant map, denoted by $S_1$,
is the connected component of $S$. Then we have the following
long exact sequence
$$ 0\rightarrow H^0(\Gamma, S_1)\rightarrow
H^0(\Gamma, S)\rightarrow H^0(\Gamma, \mathbb{Z}/2\mathbb{Z})
\rightarrow H^1(\Gamma,S_1) \rightarrow H^1(\Gamma,S) \rightarrow
H^1(\Gamma,\mathbb{Z}/2\mathbb{Z}).$$ 
Using \cite[Proposition
4.2.11]{Finte_groups_lie_type}, we get that $H^1(\Gamma,S_1)$ is
trivial. Since $\ell$ is odd,
we have $H^1(\Gamma,\mathbb{Z}/2\mathbb{Z})=\{0\}$.
Then it follows from the above exact
sequence that $H^1(\Gamma,S)$ is trivial, and so is
$H^1(\Gamma, Y(\mathbb{F}_{q^\ell}))$. Thus, the exact sequence
(\ref{exact_seq_1}) gives the equality $X_\ell^\sigma = X$.
\begin{theorem}
For any ${\rm Sp}_4(\mathbb{F}_{q^\ell})\rtimes 
{\rm Gal}(\mathbb{F}_{q^\ell}/\mathbb{F}_q)$ stable 
$\Lambda$-lattice 
$\mathcal{L}\subseteq \Pi$, the first 
Tate cohomology group $\widehat{H}^1(
{\rm Gal}(\mathbb{F}_{q^\ell}/\mathbb{F}_q), \mathcal{L})$ 
is trivial.
\end{theorem}
\begin{proof}
Since ${\rm res}_{P(\mathbb{F}_{q^\ell})}(\Pi)$ 
is isomorphic to $\rho_\ell$, it suffices 
to show that $\widehat{H}^1(\mathcal{L})$ 
is trivial 
for any $P(\mathbb{F}_{q^\ell})\rtimes 
{\rm Gal}(\mathbb{F}_{q^\ell}/\mathbb{F}_q)$ stable 
$\Lambda$-lattice 
$\mathcal{L}$ in $\rho_\ell$. Recall that 
the action of 
${\rm Gal}(\mathbb{F}_{q^\ell}/\mathbb{F}_q)$
on $\rho_\ell$ 
is defined by 
$$ (\sigma f)(x) = f(\sigma^{-1}(x)),\ \  \sigma\in 
{\rm Gal}(\mathbb{F}_{q^\ell}/\mathbb{F}_q)$$ for all
$x\in P(\mathbb{F}_{q^\ell})$ and $f\in \rho_\ell$. Let $\mathcal{L}$
be the $P(\mathbb{F}_{q^\ell})\rtimes \Gamma$ stable
lattice in $\rho_\ell$ consisting of 
$\Lambda$-valued functions in $\rho_\ell$. Since
$\widetilde{\eta_z}(\sigma(y)) = \widetilde{\eta_z}(y)$ for
$y\in Y(\mathbb{F}_q)$, the Tate cohomology group
$\widehat{H}^1(\widetilde{\eta}_z)$ is trivial.  Then it follows from the
isomorphism (\ref{Tate_restriction_map}) that
$$ \widehat{H}^1(\mathcal{L}) = 0. $$
Since the mod-$\ell$ reduction $r_\ell(\rho_\ell)$ is irreducible, any
$P(\mathbb{F}_{q^\ell})$ stable lattice is homothetic to
$\mathcal{L}$.  Thus, we get that $\widehat{H}^1(\mathcal{L})$ is
trivial for any $\Lambda[P(\mathbb{F}_{q^\ell})\rtimes \Gamma]$ stable
lattice $\mathcal{L}$ in $\rho_\ell$.
\end{proof}
\subsubsection{}
Next, we show that the zeroth
Tate cohomology of the mod-$\ell$ reduction of $\Pi$ 
is irreducible and cuspidal.
\begin{lemma}\label{irr_Sp4}
Let $\Pi$ be the irreducible unipotent cuspidal representation of
${\rm Sp}_4(\mathbb{F}_{q^\ell})$. Then the Tate cohomology
group $\widehat{H}^0(r_\ell(\Pi))$ is irreducible and cuspidal.
\end{lemma}
\begin{proof}
  Note that $r_\ell(\Pi)$ is irreducible.  Further, it follows from
  the exactness of the Jacquet functor and \cite[Chapter I, Section
  9.3]{Vigneras_modl_book} that $r_\ell(\Pi)$ is also cuspidal. By
  Corollary \ref{Tate_cuspidal}, we get that
  $\widehat{H}^0(r_\ell(\Pi))$ is cuspidal.

To prove the irreducibility of $\widehat{H}^0(r_\ell(\Pi))$, we
consider the following restriction to $P(\mathbb{F}_q)$ map
$$ \big({\rm Ind}_{Y(\mathbb{F}_{q^\ell})}
^{P(\mathbb{F}_{q^\ell})}
(r_\ell(\widetilde{\eta_z}))\big)^\sigma \longrightarrow
{\rm Ind}_{Y(\mathbb{F}_{q})}^{P(\mathbb{F}_{q})}
(r_\ell(\widetilde{\psi_z})^\ell) $$ 
$$ f \longmapsto {\rm res}_{P(\mathbb{F}_q)}(f). $$
The above map factors through $\widehat{H}^0(r_\ell(\rho_\ell))$ and
induces the following $P(\mathbb{F}_q)$ isomorphism (see Subsection
\ref{section_kirrilov_rep})
$$ \widehat{H}^0(r_\ell(\rho_\ell)) 
\simeq r_\ell(\rho)^{(\ell)}. $$ Since the restriction
${\rm res}_{P(\mathbb{F}_{q^\ell})}(r_\ell(\Pi))$ is isomorphic to
$r_\ell(\rho_\ell)$, the Tate cohomology
$\widehat{H}^0(r_\ell(\Pi))$, being $P(\mathbb{F}_q)$ isomorphic to
$r_\ell(\rho)^{(\ell)}$, is an irreducible representation of
$P(\mathbb{F}_q)$, and hence of ${\rm Sp}_4(\mathbb{F}_q)$.  This
shows the lemma.
\end{proof}
\subsection{}
Next, we show that the Tate cohomology group
$\widehat{H}^0(r_\ell(\Pi))$ is isomorphic to
$r_\ell(\pi)^{(\ell)}$. Let $Q$ be the Klingen parabolic subgroup of
${\rm Sp}_4$ consisting of matrices $(g_{ij})\in {\rm GL}_4$ with
$g_{21}=g_{31}=g_{42}=g_{43}=0$.  Let $L$ be the Levi subgroup of $Q$
and $H$ be the unipotent radical of $Q$.  Then $L(\mathbb{F}_q)$ is
isomorphic to $\mathbb{F}_q^\times \times {\rm SL}_2(\mathbb{F}_q)$,
and $H(\mathbb{F}_q)$ consists of matrices of the form
$$ h=
\begin{pmatrix}
1 & x & y & z\\
0 & 1 & 0 & y\\
0 & 0 & 1 & -x\\
0 & 0 & 0 & 1
\end{pmatrix},\, x,y,z\in\mathbb{F}_q. $$
\subsubsection{}
Let $\xi$ denote the irreducible representation of $H(\mathbb{F}_q)$
on the space $\overline{\mathbb{Q}}_\ell[\mathbb{F}_q]$ consisting of
$\overline{\mathbb{Q}}_\ell$-valued functions on $\mathbb{F}_q$,
defined by
\begin{equation}\label{action_unipotent}
(\xi(h)\varphi)(t) = \psi(z+2ty+xy)\,
\varphi(x+t),
\end{equation}
for $h\in H(\mathbb{F}_q)$ and
$\varphi \in \overline{\mathbb{Q}}_\ell [\mathbb{F}_q]$. Let
$A(\mathbb{F}_q)$ be the subgroup
$\{\pm I_4\}\,{\rm SL}_2(\mathbb{F}_q) H(\mathbb{F}_q)$, where $I_4$
is the $4\times 4$ identity matrix. The $H(\mathbb{F}_q)$
representation $\xi$ is extended to a representation of
$A(\mathbb{F}_q)$ where the actions of $\{\pm I_4\}$ and
${\rm SL}_2(\mathbb{F}_q)$ are defined as in \cite[equations $(1)$ and
$(2)$]{Golfand_Sp4}. The restriction
${\rm res}_{{\rm SL}_2(\mathbb{F}_q)}(\xi)$ splits into two
irreducible representations $\xi^{+}$ and $\xi^{-}$ acting on the
spaces of even and odd functions in
$\overline{\mathbb{Q}}_\ell[\mathbb{F}_q]$, respectively. We further
extend the representation $\xi^{-}$ trivially to the group
$A(\mathbb{F}_q)$ and set $\widetilde{\xi} = \xi\otimes \xi^{-}$.  The
induced representation
${\rm Ind}_{A(\mathbb{F}_q)}^{Q(\mathbb{F}_q)} (\widetilde{\xi})$ is
irreducible, which we denote by $\pi_1$.  Similarly, we have an
irreducible representation
$\pi_{1,\tau} = {\rm Ind}_{A(\mathbb{F}_q)}^{Q(\mathbb{F}_q)}
(\widetilde{\xi_\tau})$, where $\widetilde{\xi_\tau}$ is the
representation of $A(\mathbb{F}_q)$ defined analogously to
$\widetilde{\xi}$ by replacing the character $\psi(t)$ by
$\psi_\tau(t) = \psi(\tau t)$. Then the restriction
${\rm res}_{Q(\mathbb{F}_q)}(\pi)$ is isomorphic to the direct sum
$\pi_1 \oplus \pi_{1,\tau}$ (\cite[Proposition 2]{Golfand_Sp4}).
\subsubsection{}
Using similar constructions as in the preceding subsection, we get
that the restriction ${\rm res}_{Q(\mathbb{F}_{q^\ell})}(\Pi)$ is
isomorphic to the direct sum $\Pi_1\oplus \Pi_{1,\tau}$, where $\Pi_1$
and $\Pi_{1,\alpha}$ are the irreducible induced representations
${\rm Ind}_{A(\mathbb{F}_{q^\ell})}^
{Q(\mathbb{F}_{q^\ell})}
(\widetilde{\xi}_\ell)$ and
${\rm Ind}_{A(\mathbb{F}_{q^\ell})}^
{Q(\mathbb{F}_{q^\ell})}
(\widetilde{\xi_\tau}_\ell)$ respectively, 
and the irreducible
representations $\widetilde{\xi}_\ell$
(resp. $\widetilde{\xi_\tau}_\ell$) is defined analogously to
$\widetilde{\xi}$ (resp. 
$\widetilde{\xi_\tau}$) by replacing the
character $\psi$ by
$\eta=\psi\circ{\rm Tr}_{\mathbb{F}_{q^\ell}/\mathbb{F}_q}$.  Note
that $\Pi_1$ and $\Pi_{1,\tau}$ are endowed with an action of
$\Gamma$, given by
$$ (\sigma f)(x) = f(\sigma^{-1}(x)),\,\,
\forall x\in Q(\mathbb{F}_{q^\ell}),\, f\in \Pi_1\,(\text{or}\,\,\Pi_{1,\tau}). $$
Let $\chi_{\widetilde{\xi}_\ell}$ and
$\chi_{\widetilde{\xi_\tau}_\ell}$ be the characters of
$\widetilde{\xi}_\ell$ and $\widetilde{\xi_\tau}_\ell$ respectively.
Then using the relation (\ref{action_unipotent}) and \cite[Equations
$(1)$ and $(2)$]{Golfand_Sp4}, we get the following identity of
characters:
\begin{equation}\label{char_symp}
\overline{\chi_{\widetilde{\xi}_\ell}(g)} 
= \overline{\chi_{\widetilde{\xi}}(g)}^\ell
\,\,\,\,\,\text{and}\,\,\,\,
\overline{\chi_{\widetilde{\xi_\tau}_\ell}(g)} 
= \overline{\chi_{\widetilde{\xi_\tau}}(g)}^\ell
\end{equation}
for all $g\in A(\mathbb{F}_q)$. We now prove the following theorem.
\begin{theorem}
  Let $\ell$ be an odd prime number.  Let $\Gamma$ be the Galois group
  ${\rm Gal}(\mathbb{F}_{q^\ell}/\mathbb{F}_q)$, and let $\sigma$ be
  the Frobenius element in $\Gamma$. Let $\pi$ and $\Pi$ be the
  irreducible unipotent $\ell$-adic cuspidal representation of
  ${\rm Sp}_4(\mathbb{F}_q)$ and ${\rm Sp}_4(\mathbb{F}_{q^\ell})$,
  respectively.  Then the Tate cohomology group
  $\widehat{H}^0(r_\ell(\Pi))$ is isomorphic to the Frobenius twist
  $r_\ell(\pi)^{(\ell)}$.
\end{theorem}
\begin{proof}
Let $\chi_\Pi$ and $\chi_\pi$ be the 
characters of $\Pi$ and $\pi$ respectively. 
We will show that 
$$ \overline{\chi_\Pi(g)} = 
\overline{\chi_\pi(g)}^\ell, $$ for all
$g\in {\rm Sp}_4(\mathbb{F}_q)$.  Let $g=su=us$ be the Jordan
decomposition of $g$.  Note that either $g$ is contained in a
parabolic subgroup of ${\rm Sp}_4(\mathbb{F}_q)$ or $u$ is the
identity element and $s$ is regular.  If $s$ is a regular semisimple
element, we have $\chi_\Pi(s) = 1 = \chi_\pi(s)$.  Assume that $g$
belongs to a parabolic subgroup $P(\mathbb{F}_q)$ or
$Q(\mathbb{F}_q)$.  If $g\in P(\mathbb{F}_q)$, then it follows from
the arguments of Lemma \ref{irr_Sp4} that
$$ \widehat{H}^0(r_\ell(\Pi)) 
\simeq r_\ell(\pi)^{(\ell)} $$
as a representation of $P(\mathbb{F}_q)$. 
Thus we have
$$ \overline{\chi_\Pi(g)} =
\overline{\chi_\pi(g)}^\ell $$
for all $g\in P(\mathbb{F}_q)$. 

Suppose $g\in Q(\mathbb{F}_q)$. Note that
$\chi_\Pi = \chi_{\Pi_1} + \chi_{\Pi_{1,\tau}}$ and 
$\chi_\pi = \chi_{\pi_1} + \chi_{\pi_{1,\tau}}$.
Using character formula for 
induced representations, we have
\begin{equation}\label{symp_equ_1}
\chi_{\Pi_1}(g) = \sum_{y\in Y_\ell,\,ygy^{-1}\in A(\mathbb{F}_{q^\ell})}
\chi_{\widetilde{\xi}_\ell}(ygy^{-1})
\end{equation}
and
\begin{equation}\label{symp_equ_2}
\chi_{\pi_1}(g) = \sum_{y\in Y,\,ygy^{-1}\in
A(\mathbb{F}_q)}\chi_{\widetilde{\xi}}(ygy^{-1}),
\end{equation}
where $Y_\ell$ and $Y$ denote the right coset spaces
$A(\mathbb{F}_{q^\ell})\backslash Q(\mathbb{F}_{q^\ell})$ and
$A(\mathbb{F}_q) \backslash Q(\mathbb{F}_q)$ respectively, with
$Y_\ell^\sigma = Y$.  Since $\Gamma$ acts freely on
$Y_\ell\setminus Y$, there exists a set
$\mathcal{U}\subseteq Y_\ell\setminus Y$ such that $Y_\ell\setminus Y$
is the disjoint union of $\sigma^i\mathcal{U}$, $0\leq i\leq
\ell-1$. Also note that
$\chi_{\widetilde{\xi}_\ell}(\sigma(a)) =
\chi_{\widetilde{\xi}_\ell}(a)$ for all
$a\in\mathbb{A}(\mathbb{F}_{q^\ell})$. Then the identity
(\ref{symp_equ_1}) becomes
$$  \chi_{\Pi_1}(g) = \sum_{y\in Y,\,
ygy^{-1}\in A(\mathbb{F}_{q})}
\chi_{\widetilde{\xi}_\ell}(ygy^{-1})\,\,+\,\,
\ell\sum_{u\in\mathcal{U},\,ugu^{-1}\in
A(\mathbb{F}_{q^\ell})}
\chi_{\widetilde{\xi}_\ell}(ugu^{-1}). $$
Using the relations (\ref{char_symp})
and taking mod-$\ell$ reduction to the
above equation, we get
$$ \overline{\chi_{\Pi_1}(g)} = \sum_{y\in Y,\,
  ygy^{-1}\in A(\mathbb{F}_{q})}
\overline{\chi_{\widetilde{\xi}}(ygy^{-1})}^\ell. $$ Now, taking
mod-$\ell$ reduction to the relation (\ref{symp_equ_2}) and comparing
it with the above identity, we get
$$ \overline{\chi_{\Pi_1}(g)} =
\overline{\chi_{\pi_1}(g)}^\ell. $$
Similarly, we have
$$ \overline{\chi_{\Pi_{1,\tau}}(g)} =\overline{\chi_{\pi_{1,\tau}}(g)}^\ell. $$
Thus, we get the character identity
$$ \overline{\chi_\Pi(g)} =
\overline{\chi_\pi(g)}^\ell, $$ for all
$g\in {\rm Sp}_4(\mathbb{F}_q)$. Using \cite[Corollary
7.21]{Lam_noncommutative_rings}, we conclude that the Tate cohomology
$\widehat{H}^0(r_\ell(\Pi))$ is isomorphic to $r_\ell(\pi)^{(\ell)}$.

\end{proof}

\section{Conflict of interest statement}
On behalf of all authors, the corresponding author states that there
is no conflict of interest.

\section{Date statement}
No associated data

\bibliography{biblio} 
\bibliographystyle{amsalpha}
Sabyasachi Dhar,\\
\texttt{mathsabya93@gmail.com},
\texttt{sabya@iitk.ac.in}\\
Santosh Nadimpalli, \\
\texttt{nvrnsantosh@gmail.com}, \texttt{nsantosh@iitk.ac.in}.\\
Department of Mathematics and Statistics, Indian
Institute of Technology Kanpur, U.P. 208016, India.
\end{document}